\documentclass{article}
\usepackage{natbib}
\usepackage{amsmath}
\usepackage{amsthm}
\usepackage{amssymb}
\usepackage{amsfonts}
\usepackage{booktabs}

\usepackage{blkarray}
\usepackage{bigstrut}
\usepackage{enumitem}
\usepackage{hyperref}
\usepackage{graphicx} % Required for inserting images
\usepackage[margin=1in]{geometry}
\usepackage{comment}
\usepackage{nicematrix}
\usepackage{nicefrac}
\usepackage{mathabx}
\usepackage{circuitikz}
\usepackage{color}
\usepackage{soul}
\usepackage{proof}

\theoremstyle{definition}

\newtheorem{theorem}{Theorem}
\newtheorem{definition}{Definition}
\newtheorem{lemma}{Lemma}

\newtheorem{corollary}{Corollary}
\newtheorem{proposition}{Proposition}
\newtheorem{remark}{Remark}
\newtheorem{example}{Example}

\newcommand{\semantics}[1]{[\![\mbox{\em $ #1 $\/}]\!]}

\newcommand{\prob}{\mathbb{P}}
\newcommand{\probsymbol}{\mathsf{P}}
\newcommand{\Val}{\mathrm{Val}}
    
\usepackage{pict2e}

\makeatletter

\newcommand*{\sanssum}{\DOTSB\mathop{\mathpalette{\sans@op\sum{0.9}}\sans@sum}\slimits@}
\newcommand*{\sansprod}{\DOTSB\mathop{\mathpalette{\sans@op\prod{0.8}}\sans@prod}\slimits@}

\newcommand*{\sans@op}[4]{%
   \sbox0{\m@th$\m@th#3#1$}%
   \unitlength\ht0
   \advance\unitlength\dp0
   \vcenter{%
      \hbox{%
         \begin{picture}(#2,1)
            \linethickness{%
               \ifx#3\displaystyle2\fontdimen8\textfont\else
               \ifx#3\textstyle1.7\fontdimen8\textfont\else
               \ifx#3\scriptstyle1.6\fontdimen8\scriptfont\else
               1.6\fontdimen8\scriptscriptfont\fi\fi\fi 3}
            \roundcap
            \roundjoin
            #4
         \end{picture}%
      }%
   }%
}

\newcommand*{\sans@sum}{%
   \polyline(0.83,0.07)(0.08,0.07)(0.45,0.5)(0.08,0.93)(0.83,0.93)%
}
\newcommand*{\sans@prod}{%
   \put(0.05,0.93){\line(1,0){0.7}}%
   \put(0.19,0.07){\line(0,1){0.86}}%
   \put(0.61,0.07){\line(0,1){0.86}}%
}

\makeatother

\title{On Probabilistic and Causal Reasoning with Summation Operators}
\author{Duligur Ibeling$^*$, Thomas Icard$^*$, and Milan Mossé$^\dagger$\\
{\small $^*$ Stanford University}\\
{\small $^\dagger$ University of California, Berkeley}}
\date{May 2024}

\begin{document}
\maketitle

% Venues
% 1. Journal of logic and computation
% 2. Logical methods in computer science: https://lmcs.episciences.org
% 3. ACM transactions in computational logic
% 4. Journal of Causal Inference

% Can we call the \circ one coeff, since coefficients are restricted to +1 rather than -1 and variables?
% 
% Results for review: ``almost everywhere'' undefinability

% Milan's to do items: Mention that Blaser et al's characterization allows us to lift our results to succETR, and maybe add the proof that SAT_{L_prob} is succETR-hard.

\begin{abstract}
    \cite{IBELING2023103339} axiomatize increasingly expressive languages of causation and probability, and \cite{mosse2022causal} show that reasoning (specifically the satisfiability problem) in each causal language is as difficult, from a computational complexity perspective, as reasoning in its merely probabilistic or ``correlational'' counterpart. Introducing a summation operator to capture common devices that appear in applications---such as the $do$-calculus of \cite{Pearl2009} for causal inference, which makes ample use of marginalization---\cite{van2023hardness} partially extend these earlier complexity results to causal and probabilistic languages with marginalization. We  complete this extension, fully characterizing the complexity of probabilistic and causal reasoning with summation, demonstrating that these again remain equally difficult. Surprisingly, allowing free variables for random variable values results in a system that is undecidable, so long as the ranges of these random variables are unrestricted. We finally axiomatize these languages featuring marginalization (or more generally summation), resolving open questions posed by \cite{IBELING2023103339}.
\end{abstract}

\begin{comment}
    TO DO:
    - Does completeness work with just one set of range variables?
    - Does succETR argument work with infinite ranges?
    - Add examples for separation
    - Argument that SAT with free variables is RE

\end{comment}

\section{Introduction and Motivation}

One important research program in probability logic is to design, analyze, and compare alternative formal language fragments designed to reason about probability distributions. The usual logical tradeoff between expressive power---how much one can say---and (both formal and intuitive) complexity is very much on display in this subject. At one extreme are very simple languages that admit only of ``qualitative'' likelihood comparisons between events, which are often no more complex than propositional logic \citep{fagin1990logic, IBELING2023103339}. Toward the other extreme are first-order probability logical languages whose validity problems are often highly undecidable, not even arithmetical \citep{abadi1994decidability}. 

Summing over the values of a random variable is ubiquitous in the theory and application of probability. Marginalizing out a variable, or taking an expectation of a (discrete) variable, for example, involves summing an expression over all possible values of that variable. Summation operators, however, rarely appear in studies of probability logic. In the finite setting---for instance, where all random variables come with a fixed, finite range of possible values---this is immaterial when it comes to the expressive power of the language. Any such sum can simply be replaced be an explicit application of binary addition over all the summands. There are several limitations to this strategy, however. The first is obviously that it does not extend to the infinitary setting. Indeed, it does not even work in the setting of finite ranged variables but where those ranges (and perhaps their sizes) are unknown. That is, the variable values may not all be named by a constant in the language. A second, more subtle limitation is that summation operators may affect the complexity of the language, since they may facilitate succinct encodings of statements that would otherwise require very long expressions. As recently shown by \cite{van2023hardness}, adding a marginalization operator to a polynomial language results in a substantial complexity increase.

The present contribution aims at a thorough analysis of probability logics with summation operators, building on previous work, and focusing on questions of axiomatization and computational complexity. Part of our motivation comes not from ordinary probabilistic reasoning, but from problems in causal reasoning. In addition to the well-known logical problem of induction---the fact that any pattern of observations is logically compatible with any continuation of it---there is also a deep logical problem of causation: roughly, the fact that any pattern of observations logically underdetermines what would happen to a system when some agent acts upon the system. By capturing what is meant by ``data,'' ``causal assumptions,'' and ``causal conclusions" with formal languages, these types of result can be construed as (everywhere or \emph{almost} everywhere) logical undefinability results \citep{bareinboim2022pearl,ibeling2021topological}. The relevant formalisms commonly build on an appropriate choice of probability logical language. 

In addition to facilitating such negative results, logical formalization of causal reasoning also facilitates the analysis and verification of positive results. Seminal observations across many fields---from econometrics to artificial intelligence to epidemiology---can be cast in a rigorous manner, sometimes highlighting and clarifying hidden assumptions that might otherwise go unnoticed. Here are two well-known \emph{identifiability} results that help motivate the kinds of languages we study in this paper:
\begin{example}[Local average treatment effect] \label{example:late }Imagine a clinical trial that involves assignment ($Z$) to a control condition (value $z^-$) or a treatment condition (value $z^+$), with some outcome $Y$ of interest. The trial cannot guarantee full compliance, so there is an additional variable ($X$) for whether the treatment is in fact administered (has value $x^+$, as opposed to $x^-$).

A quantity of interest is the causal effect of the treatment ($X$) on the outcome ($Y$), among those who comply with their assignment. This is the conditional expectation $\mathbf{E}(Y_{x^+} - Y_{x^-} \mid x^+_{z^+} \land x^-_{z^-})$,\footnote{Note that $x^+_{z^+}$, etc. abbreviates the potential outcome $X_{Z = z^+} = x^+$, etc.} which in general is given by the sum, 
\begin{equation}
\sanssum_{y, y'} \, (\mathsf{y} - \mathsf{y}') \cdot \probsymbol(y_{x^+} \land y'_{x^-} \mid x^+_{z^+} \land x^-_{z^-}),\label{eqn:latedef}
\end{equation}
and thus when $Y$ is also binary (with range $\{y^- = 0, y^+ = 1\}$), simply by the difference $\probsymbol(y^+_{x^+} \mid x^+_{z^+} \land x^-_{z^-}) - \probsymbol(y^+_{x^-} \mid x^+_{z^+} \land x^-_{z^-})$. 

For various reasons researchers may not be able to track which individuals comply with their assignments. This is a typical ``Level 3'' counterfactual claim, a class that is in general difficult to ascertain empirically (see \citealt{bareinboim2022pearl,ibeling2021topological}). 
It was observed by \cite{angrist1996identification} that \eqref{eqn:latedef} can nonetheless be determined from typically available experimental data, provided two assumptions are made. First, it must be that $Z$ (assignment) affects $Y$ (the outcome) only through the variable $X$. Second, there should be no \emph{defiers}, that is, individuals who would take the treatment if and only if they are assigned to the control condition. Symbolically, using notation to be introduced in the sequel, this positive result can be viewed as a fact about entailment in suitable causal-probability logical languages:
% $$\begin{array}{l c c}
% \big[y \not\equiv y' \rightarrow \probsymbol(y_{x,z^+} \wedge y'_{x, z^-})\approx \underline{0}\big] \quad & \quad \quad &\quad  \\
% % \probsymbol(x_z) \approx \probsymbol(x)\quad & \quad  \quad &\quad \\
% \probsymbol(x^-_{z^+} \wedge x^+_{z^-}) \approx \underline{0} \quad & \quad \vDash \quad &\quad \mathbf{E}(Y_{x^+} - Y_{x^-} \mid x^+_{z^+} \land x^-_{z^-}) \approx \frac{\mathbf{E}\big(Y_{z^+, X_{z^+}} - Y_{z^-, X_{z^-}}\big)}{\mathbf{E}(X_{z^+} - X_{z^-})}.  \\ 
% % \probsymbol(y_{x,z})  \approx \probsymbol(y_z) \quad & \quad \quad &\quad \\
% % \probsymbol(y_z|x_z) \approx \probsymbol(y|x,z) \quad & \quad \quad &\quad 
% \end{array}$$
\begin{multline}
\big[y \not\equiv y' \rightarrow \probsymbol(y_{x,z^+} \wedge y'_{x, z^-})\approx \underline{0}\big], \quad 
\probsymbol(x^-_{z^+} \wedge x^+_{z^-}) \approx \underline{0}    \\
\vDash  \mathbf{E}(Y_{x^+} - Y_{x^-} \mid x^+_{z^+} \land x^-_{z^-}) \approx \frac{\mathbf{E}\big(Y_{z^+, X_{z^+}} - Y_{z^-, X_{z^-}}\big)}{\mathbf{E}(X_{z^+} - X_{z^-})}.\label{eqn:lateidentification}
\end{multline}
Any situation in which the claims appearing before the turnstile symbol $\vDash$ are all true will guarantee that the conditional expectation after the symbol is equal to this ratio. In the setting where $Y$ is binary, this can be derived in a relatively simple logical calculus for (probabilistic) causal reasoning \citep{ibeling2023}, and it can also be shown that both numerator and denominator of the ratio are identifiable.
\end{example}
\begin{example}[Front-door criterion] \label{example:front} For a second example, suppose one is interested in the probability of $Y$ taking on some value $y$ given an intervention setting variable $X$ to $x$,  notated by $\probsymbol(y_x)$. Imagine, however, that one has only passive observational access to the $X$ and $Y$. In general, it is impossible to infer $\probsymbol(y_x)$, even given perfect access to the full joint distribution $\probsymbol(X,Y)$ (see again \citealt{bareinboim2022pearl}).

Nevertheless, as shown in \cite{Pearl1995}, it is possible to determine such probabilities exactly from interventional data, provided one can make certain independence assumptions involving one more ``intermediate mechanism'' $Z$. The crucial fact is the following causal-probabilistic logical entailment:
\begin{multline*}
\probsymbol(z_x) \approx \probsymbol(z\mid x), \quad
\probsymbol(x_z) \approx \probsymbol(x),\quad
\probsymbol(y_x\mid z_x) \approx \probsymbol(y_{x,z}),\quad
\probsymbol(y_{x,z})  \approx \probsymbol(y_z), \quad
\probsymbol(y_z\mid x_z) \approx \probsymbol(y\mid x,z)\\
\vDash \probsymbol(y_x) \approx \sanssum_{z} \probsymbol(z\mid x) \cdot \sanssum_{x'}\probsymbol(y\mid x',z)\cdot\probsymbol(x').
\end{multline*}
% $$\begin{array}{l c c}
% \probsymbol(z_x) \approx \probsymbol(z|x), \quad & \quad \quad &\quad  \\
% \probsymbol(x_z) \approx \probsymbol(x),\quad & \quad  \quad &\quad \\
% \probsymbol(y_x|z_x) \approx \probsymbol(y_{x,z}), \quad & \quad \vDash \quad &\quad \probsymbol(y_x) \approx \sanssum_{z} \probsymbol(z|x) \cdot \sanssum_{x'}\probsymbol(y|x',z)\cdot\probsymbol(x').  \\ 
% \probsymbol(y_{x,z})  \approx \probsymbol(y_z), \quad & \quad \quad &\quad \\
% \probsymbol(y_z|x_z) \approx \probsymbol(y|x,z) \quad & \quad \quad &\quad 
% \end{array}$$
%$$\begin{array}{l c c}
%\probsymbol([x]z) \approx \probsymbol(z|x) \quad & \quad \quad &\quad  \\
%\probsymbol([z]x) \approx \probsymbol(x)\quad & \quad  \quad &\quad \\
%\probsymbol([x]y|[x]z) \approx \probsymbol([x,z]y) \quad & \quad \vDash \quad &\quad \probsymbol([x]y) \approx \sanssum_{z} \probsymbol([z]x) \cdot \sanssum_{x'}\probsymbol(y|x',z)\cdot\probsymbol(x').  \\ 
%\probsymbol([x,z]y)  \approx \probsymbol([z],y) \quad & \quad \quad &\quad \\
%\probsymbol([z]y|[z]x) \approx \probsymbol(y|x,z) \quad & \quad \quad &\quad 
%\end{array}$$
So long as the equalities before the turnstile all hold, the final identity is guaranteed. Pearl's \citeyearpar{Pearl1995,Pearl2009} \emph{$do$-calculus} presents a correspondence between graphical properties and independencies. For instance, all of the equalities before the turnstile will hold on the graph depicted in Figure~\ref{fig:frontdoorgraph}.
\begin{figure}[h]
\begin{center}
\begin{tikzpicture}
  \node (s0) [draw=black,circle] at (0,0) {$X$};
  
  \node (s2) [draw=black,circle] at (2,0) {$Z$};
  
  \node (s3) [draw=black,circle] at (4,0)  {$Y$};
  
  \node (s4) [circle,fill=gray!40] at (2,1.25) {$U$};
 
  \path (s0) edge[->,thick] (s2);
  
  \path (s2) edge[->,thick] (s3);
  
    \path (s4) edge[->,dotted] (s0);
  
  \path (s4) edge[->,dotted] (s3);
  
 \end{tikzpicture} %\caption{A graph over which (\ref{second})-(\ref{fourth}) are all valid.}\label{bow} 
 \end{center} 
 \caption{Front-door graph.}\label{fig:frontdoorgraph}
 \end{figure}
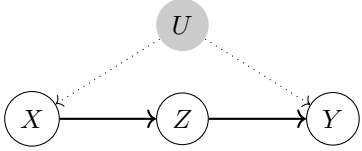 \noindent In other words, if a practitioner happens to know that a domain is one in which $X$ causes $Y$ only through an intermediate mechanism $Z$, and any further confounding ($U$) exerts influence on only $X$ and $Y$ (and in particular, not on $Z$), then the identification is valid. That is, one can determine the interventional probability knowing only the relevant observational probabilities. 
\end{example}

Both of these examples illustrate cases where some constraints on a causal setup---i.e., the expressions on the left-hand side of the turnstile---are sufficient to entail that some causal quantity of interest on the right-hand side can be identified with another expression that may be more empirically tractable. Notably, these conclusions follow as a matter of logical entailment.

A natural starting point for formalizing these styles of reasoning is to employ some form of probability logic. This is precisely the strategy pursued in recent work by \cite{ibelingicard2020,ibeling2021topological,ibeling2023,mosse2022causal,bareinboim2022pearl}. As a typical example language, we might have Boolean combinations of real polynomials over probability terms, where those probability terms include interventional probabilities like $\probsymbol(y_x)$, as well as so-called counterfactual probabilities like $\probsymbol(y^+_{x^+} \wedge y^-_{x^-})$, and so on. Provided the variable ranges are all finite, examples like those in Exs. \ref{example:late } and \ref{example:front} can be easily formalized. The works cited above investigate standard logical issues for such languages, viz. complexity, definability, axiomatization, etc. In particular, \cite{mosse2022causal} characterize the complexity of (suitably formalized) causal reasoning, thereby circumscribing the complexity of \href{https://www.causalfusion.net}{common tools} for causal reasoning.% while paving the way for general, efficient (approximate) algorithms for formalized causal reasoning \citep{deligkas2022approximating}.

While such formalization is possible in principle, it does risk distorting some of the logical and computational issues at play. The aforementioned point about succinctness in particular may be relevant. As \cite{van2023hardness} point out, supposing that sums are always expanded out to an exponentially-large explicit sum threatens to belie important complexity considerations. Indeed, it may inflate the run-time of some algorithms, e.g. for estimating conditional interventional distributions \citep{shpitser2012identification}, from linear to exponential. Such considerations motivate the study of probability logics (with and without causal primitives) that explicitly include some kind of summation or marginalization operator.

Another motivation for the setting considered here comes from the observation that much of this style of reasoning does not depend on the specific assumption that the range of each random variable is finite and fixed in advance by the signature. Indeed, the entailment facts listed above are more general, in that they hold even if we do not make any upfront assumptions about variable ranges. Thus, in addition to the points about succinctness and complexity, there are also broader logical reasons for exploring such constructions. 

Summation operators introduce a further feature into probability logic that is not typically included in many presentations, namely free and bound (range) variables. Summation is essentially a variable-binding operator, so we need to build up expressions that include free (range) variables.
The presence of such variables also suggests possible uses in which they are left free. In fact, both Examples \ref{example:late } and \ref{example:front} involve free variables both in the ``premises'' (before the turnstile) and in the ``conclusion'' (after the turnstile). The intended interpretation of these free variables is universal. Roughly speaking (cf. Def. \ref{def:entailment} below for a precise statement), when we write $\varphi \vDash \psi$, if $\varphi$ has free variables $\mathbf{x}$ and $\psi$ has free variables $\mathbf{y}$, then this is to be interpreted as $\forall \mathbf{x} \varphi \Rightarrow \forall \mathbf{y} \psi$. Thus, for instance, in Example \ref{example:front} the entailment is to be read: if for all values $x$ in the range of $X$, $y$ of $Y$, and $z$ of $Z$, the equalities all hold---viz. $\probsymbol(z_x)\approx \probsymbol(z\mid x)$, and so on---then for all $x$ and $y$ the conclusion holds. (Note that $z$ does not appear free in the conclusion.)

A first contribution of the present paper is to observe that this move alone is actually sufficient to render the resulting probability-logical system undecidable. \cite{li2023undecidability} recently proved the following fact:
\begin{theorem}[{\citealt[Thm.~3]{li2023undecidability}}]\label{theorem: conditional independence undecidable}
    The following problem is undecidable: Where $\mathbf{U}_i, \mathbf{V}_i, \mathbf{W}_i$ denote three pairwise disjoint subsets of the random variables $X_1,...,X_k$ for each $i=0,...,\ell$, determine whether the implication
    \begin{equation}
    \bigwedge_{i=1}^\ell \mathbf{U}_i \perp \mathbf{V}_i \mid \mathbf{W}_i \rightarrow \mathbf{U}_0 \perp \mathbf{V}_0 \mid \mathbf{W}_0 \label{eqn:liCI}
    \end{equation}
    holds for all jointly distributed $X_1,\dots ,X_k$ with finite support. (The same also holds if stated for all discrete $X_1, \dots, X_k$---i.e., for all those each of which has finite or countably infinite support.)
\end{theorem}
In other words, if we do not impose any restriction on the (finite) ranges of random variables, then the problem of deciding whether a set of conditional independence statements implies another is undecidable. Of course, such statements can be easily encoded in a system like the one we have been describing. For example the conditional independence \[
        \{X_{1}, \dots, X_{j}\} \perp \{Y_{1}, \dots, Y_{m}\} \mid  \{Z_{1}, \dots, Z_{n}\}
    \]
   can be expressed straightforwardly as the formula
    \begin{equation}
    \probsymbol(\mathbf{x} \land \mathbf{y} \mid  \mathbf{z}) \approx \probsymbol(\mathbf{x} \mid \mathbf{z}) \cdot \probsymbol(\mathbf{y}\mid \mathbf{z}) \label{eq:ind}
    \end{equation}
    where $\mathbf{x} = x_1 \land \dots \land x_j$, $\mathbf{y} = y_1 \land \dots \land y_m$, and $\mathbf{z} = z_1 \land \dots \land z_n$, with all these variables $x_1, y_1, z_1, \dots, x_j, y_m, z_n$ left free.
    Let $p_i$ for $i = 0,...,\ell$ denote such a statement for each triple of sets of random variables $(\mathbf{U}_i, \mathbf{V}_i, \mathbf{W}_i)$. Then the implication \eqref{eqn:liCI}
    %\[\bigwedge_{i=1}^\ell U_i \perp V_i \mid  W_i \rightarrow U_0 \perp V_0 \mid  W_0
%    \]
    is easily expressed as
    $
    \bigwedge_{i=1}^\ell p_i \vDash p_0.
    $
From Li's result, Thm. \ref{theorem: conditional independence undecidable}, it immediately follows that any language expressive enough to encode statements like (\ref{eq:ind}) will possess an undecidable validity (or satisfiability) problem under this universal interpretation of free variables. Note that this is importantly different from simply adding (e.g., universal) quantifiers to the language explicitly. See \cite{schaefer2023beyond} for a recent study of quantifier depth in the first-order theory of real numbers; it of course follows from Tarski's theorem that all such fragments remain decidable. Entailment facts like those above, in effect, involve several different types of universal quantification: they quantify over all possible values, \textit{for all possible ranges} of their variables simultaneously.

This result means in particular that the natural logical habitat for $do$-calculus---a logical language that generalizes the reasoning exhibited in Ex. \ref{example:front}, for instance---will be generally undecidable. At the same time, \cite{van2023hardness} have recently shown that probability-logical systems which assume a \emph{fixed} finite range for random variables remain decidable. Assuming there are no free variables at all, they were able to show that relatively expressive probability and ``Level 2'' causal-probability logics with summation operators are complete for the class $\text{succ}\exists\mathbb{R}$, whose prototypical problem is to decide whether a given Boolean circuit computes a satisfiable formula in the (first-order) existential theory of the real numbers. Another contribution of the present work (Theorem \ref{theorem: succETR completeness} below) is to show that this characterization extends to the full ``Level 3'' fragment encompassing all possible statements over probabilities of counterfactuals, thereby resolving a problem left open by  \cite{van2023hardness} and \cite{mosse2022causal}. The same theorem further extends these results to formulas with free variables.%, over the class of models which assigns all variables a fixed finite range.

Finally, in the last part of the paper we consider axiomatic questions for pure probabilistic (i.e., non-causal) logical languages with summation operators. Both the presence of free variables (with universal interpretation) and the use of summation present interesting logical challenges. 
A first observation is that the usual statement of the deduction theorem fails in this setting,  due to the interpretation of free variables.
\begin{remark} It does not generally hold that if $\varphi \vDash \psi$ then $\vDash \varphi \rightarrow \psi$. For instance, while $\probsymbol(x \land y) \approx \probsymbol(x)\cdot\probsymbol(y) \vDash \probsymbol(x'\land y') \approx \probsymbol(x')\cdot\probsymbol(y')$ since the ``conclusion'' merely relabels the free $x, y \mapsto x', y'$, it does not follow that $\vDash \probsymbol(x\land y) \approx \probsymbol(x)\cdot\probsymbol(y) \rightarrow \probsymbol(x' \land y') \approx \probsymbol(x')\cdot\probsymbol(y')$ for every $x, y, x', y'$.
\label{rmk:deductionfail}
\end{remark}
The laws governing summation can also be subtle, particularly when value ranges are allowed to be arbitrarily large, or even infinite. For instance, infinite sums over a probability term can be approximate from below by finite sums over increasing substitutions with constants. As is typical for probability logics, the systems we study here are not compact, so we cannot hope to obtain a strong completeness theorem with only finitary axioms. Instead, we provide strong axiomatizations with infinitary rules, building on previous work in probability logic. Enforcing the ``closed-world'' assumption that every element of the range is named by a constant, we prove strong completeness for several systems, including those with and without allowance for free variables. We close with further limitative results, as well as additional open questions.

\section{Preliminaries}\label{section: preliminaries}

This section lays out the syntax, semantics, and relevant notions from complexity for the probabilistic language $\mathcal{L}_{\text{prob}}$ and the probabilistic causal counterfactual language $\mathcal{L}_{\text{causal}}$.

%\textbf{purple}{(1) Summing over values in the denominator; (2) Conditional probabilities in the base language. (3) When does the infinite sum converge $\sum_x P(X != x \land Y = 1)$,}

\subsection{Syntax}

% You can write X = x_i for any value of i
% Summation x_i, x_j
% Constants x_1, x_2, x_3
% Big X, Big Y, X^\prime

% Language with no free variables

\subsubsection{Signatures}
Let $\mathbf{V}$ be a countable set of \emph{random variables}. For each $V \in \mathbf{V}$, let $\mathcal{C}_V$ be a countable set of \emph{constant symbols} (or constants) and let $\mathcal{V}_V$ be a countably infinite set of \emph{range variables}.
All these sets must be disjoint for distinct random variables $V \neq V' \in \mathbf{V}$, and we will often enumerate their elements explicitly. For range variables we write $\mathcal{V}_V = \{ v_1, v_2, v_3, \dots \}$
and often denote generic range variables $v, v'$, etc.~with such standing for distinct elements of $\mathcal{V}_V$.
As for constants, we write $\mathcal{C}_V = \{c^V_1, c^V_2, \dots\}$ in the case of infinitely many constants and $\mathcal{C}_V = \{c^V_1, \dots, c^V_N \}$ otherwise. Thus, for constant symbols, superscripts indicate the random variable, but we often write $c, c',$ etc.~for elements of $\mathcal{C}_V$ when context makes it clear which $V$.
Let $\mathcal{D}_V = \mathcal{C}_V \cup \mathcal{V}_V$, for each $V \in \mathbf{V}$. Elements of $\mathcal{D}_V$ are naturally denoted $d, d',$ etc.

\subsubsection{Languages}\label{sec:languages:infinite}
In this section, we build up our two languages, $\mathcal{L}_{\mathrm{prob}}$ and $\mathcal{L}_{\mathrm{causal}}$, comparing polynomials in probabilities of events, as well as some fragments thereof. We start with base atoms $V = v$, where $V$ ranges over random variables $\mathbf{V}$ and $v$ over variable symbols $\mathcal{V}_V$, as well as atoms $V = c^V_i$, for each $c^V_i \in \mathcal{C}_V$. For the remainder of this section, we assume there are infinitely many constant symbols in each $\mathcal{C}_V$. Then some base (deterministic) languages are:
\begin{align*}
    \mathcal{L}_{\mathrm{base}}^{\mathrm{prob}} &:= V=v \; | \; V = c^V_i \; | \; \lnot \mathcal{L}_{\mathrm{base}}^{\mathrm{prob}} \; | \; \mathcal{L}_{\mathrm{base}}^{\mathrm{prob}}\land \mathcal{L}_{\mathrm{base}}^{\mathrm{prob}}, \\
    \mathcal{L}_{\mathrm{int}} &:= \top \;|\; V=v \; | \; V = c^V_i \;| \;\mathcal{L}_{\mathrm{int}} \land \mathcal{L}_{\mathrm{int}}.
\end{align*}
Define the base causal conditional language as follows:
\begin{align*}    \mathcal{L}_{\mathrm{base}}^{\mathrm{causal}} &:= [\mathcal{L}_{\mathrm{int}}] \mathcal{L}_{\mathrm{base}}^{\mathrm{prob}}  \; | \; \neg \mathcal{L}_{\mathrm{base}}^{\mathrm{causal}}  \; | \; \mathcal{L}_{\mathrm{base}}^{\mathrm{causal}}  \land \mathcal{L}_{\mathrm{base}}^{\mathrm{causal}}. %\\
%    \mathcal{L}_{\mathrm{cond}}^{\mathrm{causal}} &:= \top \;|\; [\mathcal{L}_{\mathrm{int}}] \mathcal{L}_{\mathrm{cond}}^{\mathrm{prob}} \;|\; \mathcal{L}_{\mathrm{cond}}^{\mathrm{causal}} \lor \mathcal{L}_{\mathrm{cond}}^{\mathrm{causal}}
\end{align*}
Thus, instead of writing, e.g., $Y_{X=x}=y$ or its abbreviation $y_x$ for the interventional conditional, as in Examples \ref{example:late } and \ref{example:front}, we will adopt dynamic logical notation and write $[X=x]Y=y$, or abbreviatedly, $[x]y$.
% Define $\mathcal{L}_{\mathrm{cond}}^{\mathrm{causal}}$ as the smallest fragment of $\mathcal{L}_{\mathrm{base}}^{\mathrm{causal}}$ closed under $\lor$ and containing $\top$ as well as any formula of the form $[\alpha]\beta$ for $\alpha \in \mathcal{L}_{\mathrm{int}}$, $\beta \in \mathcal{L}_{\mathrm{base}}^{\mathrm{prob}}$ such that for every $X \in \mathbf{V}$, an atom of the form $X = c_i$ or $X = x_i$ occurs at most once in $[\alpha]\beta$.
 
Probability terms are defined as follows for $* \in \{\text{prob}, \text{causal}\}$ and any $c^V_i \in \mathcal{C}_V$, $v_i \in \mathcal{V}_V$: %$\delta' \in \mathcal{L}_{\mathrm{cond}}^*$:
\begin{align}
    \mathsf{t}^*_{\circ} &:=  \probsymbol\big(\mathcal{L}_{\mathrm{base}}^* \;|\; \mathcal{L}_{\mathrm{base}}^*\big) \; | \; \sanssum_{v_i} \mathsf{t}^*_{\circ} \; | \; \mathsf{t}^*_{\circ} + \mathsf{t}^*_{\circ} \; | \; \mathsf{t}^*_{\circ} \cdot \mathsf{t}^*_{\circ}\label{eqn:circterms}\\ % \tb{\; | \; - \mathsf{t}^*_{\circ}}
    \mathsf{t}^* &:= \probsymbol\big(\mathcal{L}_{\mathrm{base}}^* \;|\; \mathcal{L}_{\mathrm{base}}^*\big) \; | \; \sanssum_{v_i} \mathsf{t}^* \; | \; \mathsf{t}^* + \mathsf{t}^* \; |  \;\mathsf{t}^* \cdot \mathsf{t}^*\; | \;  - \mathsf{t}^* \; |\;\mathsf{c}^V_i \;|\; \mathsf{v}_i\;\label{eqn:terms}
\end{align}
% \tb{As for infinite valuations in sums, we should define $0 \cdot \infty = 0$ in order to make $\mathsf{Dist (mult.1)}$ hold. But we should be careful not to specify $\infty - \infty = 0$ since e.g. $\sum_x (2-1) = \sum_x 1 = \infty$ but it would make $\sum_x (2-1) = \sum_x 2 - \sum_x 1 = \infty - \infty = 0$. Instead, the claim that $\sum_x (2-1) = \text{something}$ is a paraphrase of $\sum_x 2 = \text{something} + \sum_x 1$, which is clear because we did not include negation in the language anyway. (If we want to avoid infinite valuations altogether we could allow only one of the two connectives $\lnot, \lor$ within their propositional parts and requiring the sum variable to occur free inside.)}

% \textcolor{blue}{Only give semantics for $t_{\text{full}^*}$ in the finite case?}

\noindent
The probability primitives are thus the conditional probabilities of base expressions.
Now, define languages comparing probability terms and equality of symbols for $* \in \{\text{prob}, \text{causal}\}$:
\begin{align}
    \mathcal{L}_* &:= \mathcal{D}_V \equiv \mathcal{D}_V \;|\; \mathsf{t}^* \succsim \mathsf{t}^* \; | 
    \; \lnot \mathcal{L} \; | \; \mathcal{L} \land \mathcal{L}. \label{eqn:recursivefinallanguage}
\end{align}
Thus, e.g., \eqref{eqn:lateidentification} is in $\mathcal{L}_{\text{causal}}$, with $y \not\equiv y'$ representing disequality of $y, y' \in \mathcal{V}_Y$.
We specify some fragments of these languages for $* \in \{\text{prob}, \text{causal}\}$:
\begin{itemize}
    \item 
    $\mathcal{L}_*^{\text{closed}}$ is the \textit{closed fragment} of $\mathcal{L}_*$ containing no free variables, i.e. in which every range variable $v$ is bound by a summation operator $\sanssum_{v}$.
    \item 
    $\mathcal{L}_*^{\circ}$ is the fragment of $\mathcal{L}_*$ with \textit{restricted coefficients}, formed by restricting the terms $\mathsf{t}^*$ in the second case of \eqref{eqn:recursivefinallanguage} to the $\mathsf{t}_\circ^*$ defined by \eqref{eqn:circterms}, i.e. those in which all coefficients are $1$ and not $-1$, $\mathsf{c}_i^V$, or $\mathsf{v}_i$.
    \item 
    $\mathcal{L}_*^{\circ+\text{closed}}$ is their intersection.
\end{itemize}
In sum, the targets of our analysis are the eight languages $\mathcal{L}^\dagger_*$ formed for each $* \in \{\text{prob}, \text{causal}\}$ and $\dagger \in \{\varepsilon, \text{closed}, \circ,  \circ +\text{closed}\}$.\footnote{Here $\varepsilon$ is the empty string.}
% Finally, where $\mathcal{L}$ represents any of the languages defined above and $n \in \mathbb{N}$, let $\mathcal{L}(n)$ denote the fragment of $\mathcal{L}$ containing formulas $\varphi$ such that $i \le n$ for any atom of the form $X = c_i$ appearing in $\varphi$. \tb{Is this needed?}

\subsubsection{Bounded Languages}\label{sec:languages:bounded}
As we alluded to above, we also consider languages whose signatures have finitely many constants, in which for some $N \ge 1$ we have
$\mathcal{C}_V = \{c_1^V, \dots, c_N^V\}$ for each $V$ (note here that $N$ is uniform over all $V$).
Where $\mathcal{L} = \mathcal{L}^\dagger_*$ represents any of the eight languages defined in \S{}\ref{sec:languages:infinite}, let $\mathcal{L}(N)$ be that constructed analogously to $\mathcal{L}$ but over a signature built from this bounded constant set.
We can equivalently consider $\mathcal{L}(N)$ to be a fragment of $\mathcal{L}$ containing all and only those $\varphi$ such that an atom  $V = c^V_i$ appears in $\varphi$ only when $i \le N$.

\subsubsection{Abbreviations and expressibility}
Given some $\delta \in \mathcal{L}_{\mathrm{base}}^{*}$ we use $\probsymbol(\delta)$ as shorthand for $\probsymbol(\delta \;|\; \top)$.
Given a term $\mathsf{t}$, $V \in \mathbf{V}$, $v \in \mathcal{V}_V$, and $\epsilon \in \mathcal{L}_{\mathrm{base}}^{\mathrm{prob}}$ we use the notation $\mathsf{t}[V = v / \epsilon ]$ to denote $\mathsf{t}$ but with the base formula $\epsilon$ substituted in place of every free occurrence of $V  = v$ (as an atom within a probability operator $\probsymbol$); given a formula $\varphi \in \mathcal{L}_*$ we write $\varphi[v / d]$, to denote $\varphi$ except with every free occurrence of $v$ (both across equalities $\equiv$ and as atoms $V = v$ in probability terms) replaced by $d \in \mathcal{D}_V$.
Informal sums with serifs $\sum$ (as opposed to formal sans-serif sums $\sanssum$) are used along with dummy variables as ``macros'' to abbreviate polynomial terms, e.g. $\sum_{i = 1}^2 \probsymbol(V = c^V_i)$ is shorthand for $\probsymbol(V = c^V_1) + \probsymbol(V = c^V_2)$.

Other comparisons of terms are readily expressible in terms of our one, e.g., $\mathsf{t}_1 \precsim \mathsf{t}_2 \Leftrightarrow \mathsf{t}_2 \succsim \mathsf{t}_1$ or $\mathsf{t}_1 \prec \mathsf{t}_2 \Leftrightarrow \mathsf{t}_2 \succsim \mathsf{t}_1 \land \lnot(\mathsf{t}_1 \succsim \mathsf{t}_2)$.  We write $\mathsf{t}_1 \approx \mathsf{t}_2$ for $\mathsf{t}_1 \succsim \mathsf{t}_2 \land \mathsf{t}_1 \precsim \mathsf{t}_2$.
Subtraction is expressible: we write $\mathbf{t}_1 - \mathbf{t}_2$ for $\mathbf{t}_1 + (- \mathbf{t}_2)$.
% Likewise, %subtraction and
Ratios are also expressible: %$\mathsf{t}_1 - \mathsf{t}_2 \succsim \mathsf{t}_3 \Leftrightarrow \mathsf{t}_1 \succsim \mathsf{t}_2 + \mathsf{t}_3$ and
$\mathsf{t}_1 / \mathsf{t}_2 \succsim \mathsf{t}_3 \Leftrightarrow \mathsf{t}_1 \succsim \mathsf{t}_2 \cdot \mathsf{t}_3$. In sums of ratios, etc., the latter pattern generalizes via clearing common denominators.
The only case in which it fails is when the bound variable of a sum occurs in a denominator, as in $\sanssum_v \mathsf{t}_1/\probsymbol(V = v) \succsim \mathsf{t}_2$, since the range of a formal summation is unbounded in our semantics below.
Thus---being careful to eschew syntactically this single exception---we may avail ourselves of % subtraction and
division as well.

Finally, although the above language does not include numerical constants, it is natural to define a term $\underline{1}$ for $\probsymbol(\top)$ and $\underline{0}$ for $\probsymbol(\bot)$.
For any $q \in \mathbb{Q}$ we can use $\underline{q}$ as a term freely, unpacking it by applying the rules above to $\underline{1}$s.
For example, $\mathsf{t} \succsim \underline{1/2}$ is short for $\mathsf{t} \cdot \big[\probsymbol(\top) + \probsymbol(\top)\big] \succsim \probsymbol(\top)$.% We can define $\underline{\infty} = \sum_{x_i} \probsymbol(\top)$.
% extendible to algebraics but at the cost of using more variables

\subsection{Semantics}

\subsubsection{Structural causal models}

The semantics for all of these languages will be defined relative to  \emph{structural causal models}, which can be understood as a very general framework for encoding \emph{data-generating processes}. In addition to the endogenous variables $\mathbf{V}$, structural causal models also employ \emph{exogenous variables} $\mathbf{U}$ as a source of random variation among endogenous settings. For extended introductions, see, e.g., \cite{Pearl2009,bareinboim2022pearl}.

%A model is a recursive and measurable structural causal model. 

\begin{definition}
A \textit{structural causal model} (SCM) $\mathfrak{M}$ is a tuple $\mathfrak{M} = (\mathcal{F}, \prob, \mathbf{U}, \mathbf{V})$, with: \begin{enumerate}[label=(\alph*)]
    \item $\mathbf{V}$ a set of \textit{endogenous variables}, with each $V \in \mathbf{V}$ taking on possible values from a range $\text{Val}(V)$,
    \item $\mathbf{U}$ a set of \textit{exogenous variables}, with each $U \in \mathbf{U}$ taking on possible values $\text{Val}(U)$,
    \item $\mathcal{F} =\{f_V\}_{V \in \textbf{V}}$ a family of \emph{structural functions}, such that $f_V$ determines the value of $V$ given the values of the exogenous variables $\mathbf{U}$ and those of the other endogenous variables $V' \in \mathbf{V}\setminus \{V\}$, and
    \item $\prob$ a probability measure on a $\sigma$-algebra $\sigma(\mathbf{U})$ on $\mathbf{U}$.
\end{enumerate}

% Every $\mathrm{Val}(X)$ and $\text{Val}(U)$ is a discrete (possibly infinite) space, for example $\mathbb{N}$ or a countably infinite partition of the reals. %\tb{Or, they can be any space, and we have a countably infinite partition to get a val sets, $\{1, 2, 3, \dots\}$. Then a sigma algebra generated by the partitions and measure.}
\end{definition}

In addition, we adopt the common assumption that our SCMs are \emph{recursive} or acyclic.
Where $\mathbf{S}$ is a set of variables let $\Val(\mathbf{S}) = \bigtimes_{S \in \mathbf{S}} \Val(S)$ be their joint range.
\begin{definition}
A SCM $\mathfrak{M}$ is {recursive} if there is a well-order $\prec$ on $\mathbf{V}$ such that $\mathcal{F}$ respects $\prec$ in the following sense: for any $V \in \mathbf{V}$, whenever the joint ranges $\mathbf{v}_1,\mathbf{v}_2\in \Val(\mathbf{V})$ have the property that $\pi_{V'}(\mathbf{v}_1) = \pi_{V'}(\mathbf{v}_2)$\footnote{Here $\pi_{V'}: \Val(\mathbf{V}) \to \Val(V')$ is the projection map from joint to single ranges, taking $\{v_V\}_{V \in \mathbf{V}}\mapsto v_{V'}$.} for all $V' \prec V$, we are guaranteed that $f_V(\mathbf{v}_1,\mathbf{u}) = f_V(\mathbf{v}_2,\mathbf{u})$ for every $\mathbf{u} \in\Val(\mathbf{U})$.
\end{definition}

Causal interventions represent the result of a \emph{manipulation} to the causal system, and  are defined in the standard way (e.g., \citealt{spirtes2000causation, Pearl2009}):
\begin{definition}
An \textit{intervention} \textbf{X}=\textbf{x} assigns a set of random variables $\textbf{X}\subset \mathbf{V}$ to a set of values $\textbf{x} \in \Val(\mathbf{X})$. An intervention induces a mapping of a system of functions $\mathcal{F} = \{f_V\}_{V \in \textbf{V}}$ to another system $\mathcal{F}_{\textbf{X=x}}$, which is identical to $\mathcal{F}$, but with $f_V$ replaced by the constant function $f_X(\cdot) = \pi_X(\mathbf{x})$ for each $X \in \textbf{X}$. Similarly, where $\mathfrak{M}$ is a model with equations $\mathcal{F}$, we write $\mathfrak{M}_{\textbf{X=\textbf{x}}}$ for the model which is identical to $\mathfrak{M}$ but with the equations $\mathcal{F}_{\textbf{X}=\textbf{x}}$ in place of $\mathcal{F}$.
\end{definition}

In order to guarantee that interventions lead to a well-defined semantics, we work with structural causal models which are measurable:

\begin{definition}
A SCM $\mathfrak{M}$ is \textit{measurable} if under every finite intervention $\mathbf{X}=\mathbf{x}$, the joint distribution $\prob(\mathbf{V})$ associated with the model $\mathfrak{M}_{\mathbf{X}=\mathbf{x}}$ is well-defined. This, in turn, defines a probability $\mathbb{P}_\mathfrak{M}(\delta)$ for $\delta \in \mathcal{L}^*_{\text{base}}$. (For details, see §2.2.2 of \citealt{mosse2022causal}.)
\end{definition}

Lastly, we restrict attention to models in which variables take values in the natural numbers:

\begin{definition}
    A SCM is \textit{countable} if the range $\Val(V)$ of any variable $V$ is a subset of $\mathbb{Z}^+ = \{1, 2, 3, \dots\}$.
\end{definition}

\subsubsection{Interpretation of terms and truth definitions}

\begin{definition}
    A \emph{model} $\mathfrak{M}$ is a recursive, measurable, countable SCM together with an assignment $c^{\mathfrak{M}} \in \Val(V)$ for each constant symbol $c \in \mathcal{C}_V$. We restrict attention to models $\mathfrak{M}$ that are \textit{closed-world} or ``Herbrand,'' such that for every $V \in \mathbf{V}$, the map $\mathcal{C}_V \to \Val(V)$ taking $c \mapsto c^{\mathfrak{M}}$ is surjective. %It follows from these assumptions that $\mathrm{Val}(U), \mathrm{Val}(V) \subseteq \mathbb{N}$, for all random variables $U, V$

\end{definition}

\begin{definition}
    A (range) variable assignment for $V \in \mathbf{V}$ is a partial mapping $\iota : \mathbf{V}\times\mathbb{Z}^+ \to \mathrm{Val}(V)$, neither surjective nor injective in general. The value $\iota(V, i)$ is the interpretation of the range variable $v_i$.
\end{definition}

% General class: no retrictions
% Herbrand setting: every constant denotes a distinct object, and every object is named by a constant (not used)
% Closed world: surjective constants (semantic)
% Finite range: semantic
%
% High complexity: finite range (maybe), not closed world
% For completeness, closed world and /or finite range (both semantic)
% For succETR: no free variables (syntactic), finite range (semantic), closed world (semantic) (motivation: van der Zander)

% Needed: expressive hierarhcy

\begin{definition} The relation $\mathfrak{M},\iota \vDash \varphi$ for $\varphi \in \mathcal{L}_{\text{causal}}$ is defined as follows.
Numerical denotations for terms \eqref{eqn:terms} are elements of $\mathbb{R} \cup \{ \infty, -\infty, \bot\}$.
We define $\semantics{\probsymbol(\delta \mid \delta')}^{\mathfrak{M}}_\iota = \mathbb{P}_\mathfrak{M}(\delta\land \delta')/\mathbb{P}_{\mathfrak{M}}(\delta')$ provided $\mathbb{P}_{\mathfrak{M}}(\delta') \neq 0$, and 
$\semantics{\probsymbol(\delta \mid \delta')}^{\mathfrak{M}}_\iota = \bot$ otherwise.
We also define $\semantics{\mathsf{c}^V_i}^{\mathfrak{M}}_\iota = ({c}^V_i)^{\mathfrak{M}}$ and $\semantics{\mathsf{v}_i}^{\mathfrak{M}}_\iota = \iota(V, i)$; this works since these are values in $\Val(V) \subset \mathbb{Z}^+\subset \mathbb{R}$.
As expected, we define $\semantics{\mathsf{t} \oplus \mathsf{t}'}^{\mathfrak{M}}_\iota = \semantics{\mathsf{t}}^{\mathfrak{M}}_\iota \oplus \semantics{\mathsf{t}'}^{\mathfrak{M}}_\iota$ for each $\oplus \in \{ \cdot, + \}$ and $\semantics{-\mathsf{t}}^{\mathfrak{M}}_\iota = -\semantics{\mathsf{t}}^{\mathfrak{M}}_\iota$.
Binary operations on our doubly-extended line are defined as expected; also
\begin{align*}
0 \cdot \infty = \infty \cdot 0 = 0\cdot (-\infty) = (-\infty) \cdot 0 = 0,\quad \infty + (-\infty) = (-\infty) + \infty = 0
\end{align*}
and any operation with $\bot$ as one of the operands yields $\bot$.
Now we define\footnote{In this definition, $\iota[V, i \mapsto n]$, for $V \in \mathbf{V}$, $i \in \mathbb{Z}^+$, $n \in \Val(V)$, denotes the map $\iota'$ such that $\mathrm{dom}(\iota') = \mathrm{dom}(\iota) \cup \{ (V, i) \}$ and $\iota'(V, i) = n$ but $\iota'(V', j) = \iota(V', j)$ for any $(V', j) \in \mathrm{dom}(\iota)$ such that $(V', j) \neq (V, i)$.}
\begin{align*}
\semantics{\sanssum_{v_i} \mathsf{t}}^{\mathfrak{M}}_\iota = \sum_{n \in \Val(V)} \semantics{\mathsf{t}}^{\mathfrak{M}}_{\iota[V, i \mapsto n]}
\end{align*}
with this sum taking the expected extended value if divergent (and $\bot$ if any summand is $\bot$). Note that the syntax of terms $\mathsf{t}$ guarantees that any divergent sum will tend toward $\infty$ or $-\infty$.

Now we define the semantics of formulas \eqref{eqn:recursivefinallanguage} as follows.
%\begin{enumerate}    
%    \item
First, $\mathfrak{M}, \iota \vDash d \equiv d'$ for $d, d' \in \mathcal{D}_V$ iff $\semantics{d}^{\mathfrak{M}}_\iota = \semantics{d'}^{\mathfrak{M}}_\iota$,
where the denotation in ranges $\Val(V)$ of elements of $\mathcal{D}_V$ for each $V \in \mathbf{V}$ are given in the obvious way:  $\semantics{c^V_i}^{\mathfrak{M}}_{\iota} = (c^V_i)^{\mathfrak{M}}$ and $\semantics{v_i}^{\mathfrak{M}}_{\iota} = \iota(V, i)$.
%    \item 
Then, 
    $\mathfrak{M}, \iota \vDash \mathsf{t} \succsim \mathsf{t}'$ iff $\semantics{\mathsf{t}}^{\mathfrak{M}}_\iota \ge \semantics{\mathsf{t}'}^{\mathfrak{M}}_\iota$, with neither $\semantics{\mathsf{t}}^{\mathfrak{M}}_\iota$ nor $\semantics{\mathsf{t}'}^{\mathfrak{M}}_\iota$ being $\bot$, and the order on the doubly-extended real line being defined as expected.   
    % \item 
    % $\semantics{\mathsf{t}\cdot \mathsf{t}^\prime}^{\mathfrak{M}}_{\iota} = \semantics{\mathsf{t}}^{\mathfrak{M}}_{\iota} \cdot \semantics{\mathsf{t}^\prime}^{\mathfrak{M}}_{\iota}$.
    % \item 
    % $\semantics{\mathsf{t}+ \mathsf{t}^\prime}^{\mathfrak{M}}_{\iota} = \semantics{\mathsf{t}}^{\mathfrak{M}}_{\iota} + \semantics{\mathsf{t}^\prime}^{\mathfrak{M}}_{\iota}$
%    \item 
Next, $\mathfrak{M},\iota \vDash  \neg \varphi $ iff $\mathfrak{M},\iota \not\vDash  \varphi$.\footnote{Thus, we have $\mathfrak{M},\iota \vDash  \neg \varphi $ if $\varphi$ contains {any} term evaluating to $\bot$, i.e., any probability conditional on a measure zero.}
%    \item 
    Finally, $\mathfrak{M},\iota \vDash  \varphi \land \psi$ iff $\mathfrak{M},\iota \vDash  \varphi$ and $\mathfrak{M},\iota \vDash  \psi$.
    % \item 
    % For $\mathsf{t}_1, \mathsf{t}_2$ which are $\sanssum$-free, $\mathfrak{M},\iota \vDash \mathsf{t}_1 \succsim \mathsf{t}_2$ iff $\semantics{\mathsf{t}_{1}^\prime}^{\mathfrak{M}}_{\iota} \geq \semantics{\mathsf{t}_2^\prime}^{\mathfrak{M}}_{\iota} $, where $\mathsf{t}_{1}^\prime$ and $\mathsf{t}_{2}^\prime$ result from clearing denominators of all conditional probabilities in $\mathsf{t}_1$ and $\mathsf{t}_2$, and adding any negative terms to both sides.

    % \item
    % For $\mathsf{t}_1$ and $\mathsf{t}_2$ not necessarily $\sanssum$-free, $\mathfrak{M},\iota \vDash \mathsf{t}_1 \succsim \mathsf{t}_2$ iff for all $\epsilon$ there exists an $N$ such that for all $n \geq N$, we have $\mathfrak{M}, \iota \vDash \mathsf{t}_1|_{n} + \epsilon \succsim \mathsf{t}_2|_{n}$, where $\mathsf{t}_1|_{n}$ replaces each sum $\sanssum_{x_i} \mathsf{t}(x_i)$ in $\mathsf{t}_1$ with the partial sum $\mathsf{t}(c_X^1) + ... + \mathsf{t}(c_X^n)$, where $c_X^1,...,c_X^n$ denote the first $n$ elements of $\text{Val}(X)$. In other words, $\mathfrak{M},\iota \vDash \mathsf{t}_1 \succsim \mathsf{t}_2$ iff $\lim_{n \rightarrow \infty} \mathfrak{M}, \iota \vDash \mathsf{t}_1|_{n}  \succsim \mathsf{t}_2|_{n}$ is true.    
% \end{enumerate}
% Define $\underline{0} \cdot \underline{\infty} =\underline{0}$. % \textcolor{blue}{What about $\infty-\infty$?}
\end{definition}

\begin{definition}[Entailment] \label{def:entailment} We say that $\varphi$ is valid in a model $\mathfrak{M}$, and write $\mathfrak{M}\vDash \varphi$, if $\mathfrak{M},\iota \vDash \varphi$ for every $\iota$.
A model $\mathfrak{M}$ satisfies a \emph{sequent} $\Gamma \Rightarrow \varphi$ if either $\mathfrak{M}\not\vDash \gamma$ for some $\gamma \in \Gamma$, or $\mathfrak{M} \vDash \varphi$.
We say that a set $\Gamma$ of formulas entails a formula $\varphi$, and write $\Gamma \vDash \varphi$, when $\mathfrak{M}$ satisfies $\Gamma \Rightarrow \varphi$ for every $\mathfrak{M}$ (in some class $\mathcal{M}'$).
\end{definition}

\begin{definition}\label{def:models}
    Let $\mathcal{M}$ be the class of (recursive, measurable, countable, closed-world) models.
    Let $\mathcal{M}_N \subset \mathcal{M}$ be the subclass in which each variable $V$ has a range of cardinality $N$.
Let $\mathcal{M}_{\text{fin}} = \bigcup_{N=1}^\infty \mathcal{M}_N$ be the subclass of those in which all variables have a finite range.%.\footnote{Thus $\mathcal{M}_{\text{fin}} = $.} %named by the constants $c_1^X,\dots c_n^X$, so that  % A model $\mathfrak{M}$ is \textit{finitely, explicitly ranged} if it assigns each variable $X$ a finite range, which is named by constants $c_X^1, c_X^2,\dots, c_X^n$. 
% \tb{Do we also need to assume these are named by $c_1, \dots, c_n$?}
\end{definition}

\begin{definition}[Satisfiability] 
    $\mathsf{SAT}_{\mathcal{L}_*^\dagger}$ is the problem of deciding whether a given sequent $\Gamma \Rightarrow \varphi$, where $\varphi \in \mathcal{L}_*^\dagger$ and $\Gamma \subset \mathcal{L}_*^\dagger$ is finite, has a model satisfying it within a fixed class $\mathcal{M}'$.
\end{definition}
% Add assumption: You know which constants refer to which values.
\citet{van2023hardness} restrict attention to $\mathsf{SAT}_{\mathcal{L}_*^{\text{closed}}(N)}$\footnote{Strictly speaking, they consider satisfiability for the fragment of $\mathcal{L}_*^{\text{closed}}(N)$ in which the primitives $\mathsf{c}^V_i$ and $\mathsf{v}_i$ in \eqref{eqn:terms} are disallowed, but this of course only means our complexity results in \S{}\ref{section: complexity} are stronger.} over $\mathcal{M}_N$.

\subsubsection{Examples for separation in expressive power}

% P(X = x) > 0 -> P(X = c) > 0

In this section, we show that the distinctions among languages and classes of models just introduced track gaps in expressive power.

\paragraph{Free and bound range variables.} Consider the following formula with a free range variable:
\begin{align*}
    \probsymbol(V=v_i) \succ \underline{0}.
\end{align*}
This formula $\varphi$ is satisfied only by models which assign positive probability to $V=n$ for every value $n \in \text{Val}(V)$. However, no formula $\psi$ with only bound range variables picks out precisely such models.% since for any closed formula, one can add elements to the range of $V$ which receive 0 probability, without affecting its satisfiability.

\paragraph{Coefficients.}
Unlike $\mathcal{L}_{\text{prob}}^{\circ+\text{closed}}$, the language $\mathcal{L}_{\text{prob}}^{\text{closed}}$ allows the use of constants, as in the following formula, where $i$ is shorthand for the $i$-term sum $\mathsf{P}(\top) + \dots +\mathsf{P}(\top) $:
\[
\bigwedge_{1 \le i \le n} c^V_i \approx i \land  \sanssum_{v_i} \mathsf{P}(\top) \approx n % probsymbol(\vee_{1\le i \le n} V= c^V_i) \approx \underline{1} \land
\]
%\[
%\bigwedge_{1 \le i \neq j \le n} c^V_i \not\equiv c^V_j \land \bigwedge_{i=1}^n \probsymbol(V= c^V_i) \approx \underline{\frac{1}{n}} \land \sanssum_{v_i} \mathsf{v}_i \cdot \probsymbol(V = v_i) \approx \underline{\frac{1}{n} \cdot \sum_{i=1}^n i}
%\]
Indeed, this formula $\varphi$ is satisfied only by models which assign $\text{Val}(V) = \{1, \dots, n\}$. But no formula $\psi \in \mathcal{L}_{\text{prob}}^{\text{closed}}$ is true of the same models as $\varphi$, since for any model $\mathfrak{M}$ of $\psi$, we can create another model $\mathfrak{M}^\prime$ where $\Val(V) = \{2, \dots, n+1\}$, associating the event $V = i$ in $\mathfrak{M}$ with $V = i+1$ in $\mathfrak{M}'$, without sacrificing the satisfaction of $\psi$.

\begin{comment}
    
\begin{lemma}
    Suppose $\varphi \in \mathcal{L}_{\text{prob}}^{\text{closed}}$. Then $\varphi$ has a model which assigns positive probability to at most $\text{poly}(|\varphi|)$ events $V=v$ such that $v$ does not appear in $\varphi$.
\end{lemma}

% Consider any formula $\varphi \in \mathcal{L}_{\text{prob}}$ without any free range variables, which is satisfied by some model $\mathfrak{M}$. We claim that $\mathfrak{M}^\prime \vDash \varphi$, for some model $\mathfrak{M}^\prime$ for which $\mathbb{P}_{\mathfrak{M}^\prime}(V= n) < \epsilon$ for some $n \in \text{Val}(V)$. Indeed, let $\mathfrak{M}^\prime$... This won't work. :(

\paragraph{Summation}

- Marginaliztion and summation (with infinite ranges). % sum_x x < N. can you show that this can't be defined in the polynomial language. add dummy things to x's range. positive support
\end{comment}

\paragraph{Probabilistic and causal languages.} To illustrate the expressivity of causal as opposed to purely probabilistic languages, we recall a variation by \cite{bareinboim2022pearl} on an example due to \cite{Pearl2009}:

\begin{example}[Causation without correlation]\label{ex-cht}
Let $\mathfrak{M}_1 = (\mathcal{F}, \prob, \mathbf{U}, \mathbf{V})$, where $\textbf{U}$ contains two binary variables $U_1,U_2$ such that $\prob(U_1) = \prob(U_2) =\nicefrac{1}{2}$, and $\textbf{V}$ contains two variables $V_1, V_2$ such that $f_{V_1} = U_1$ and $f_{V_2} = U_2$. Then $V_1$ and $V_2$ are independent. But having observed this, one could not conclude that $V_1$ has no causal effect on $V_2$; consider the model $\mathfrak{M}^\prime$, which is like $\mathfrak{M}$, except with the mechanisms:
\begin{align*}
    f_{V_1} &= \mathbf{1}_{U_1 = U_2}\\
    f_{V_2} &= U_1 + \mathbf{1}_{V_1=1, U_1 = 0, U_2=1}.
\end{align*} 
In this case $\prob_{\mathfrak{M}}(V_1,V_2) = \prob_{\mathfrak{M}^\prime}(V_1,V_2)$, so that the models are indistinguishable in $\mathcal{L}_{\mathrm{prob}}$. However, the models are distinguishable in $\mathcal{L}_{\mathrm{causal}}$ (given appropriate constant interpretations), via the statement
\begin{align*}
\prob\big([V_1 = 1] V_2 = 1\big) = \prob\big( [V_1 = 1] V_2 = 0\big)
\end{align*}
%The first of these statements is equivalent to conjunction of the statements
%\begin{align*}
%    \prob([V_1 = 1] V_2 = 1) = \prob(V_3 = 1)\\
%    \prob(V_3 = 1) = \prob(\neg (V_3 = 1)),
%\end{align*}
%and both of these statements belong to $\mathcal{L}_{\mathrm{causal}}^{\mathrm{comp}}$.
\end{example}

As shown in \cite{bareinboim2022pearl} (cf. also \citealt{suppes:zan81,ibeling2021topological}), the pattern in Example \ref{ex-cht} is universal: for any model $\mathfrak{M}$ it is \emph{always} possible to find some $\mathfrak{M}^\prime$ that agrees with $\mathfrak{M}$ on all of $\mathcal{L}_{\mathrm{prob}}$ but disagrees on $\mathcal{L}_{\mathrm{causal}}$.%Bareinboim, Correa, Ibeling \& Icard \cite{bareinboim2020pearl} show that this case is typical; in a precise sense, the most expressive causal language \textit{almost never} collapses into the most expressive probabilistic one:

\subsection{Complexity}

In this section, we introduce the target complexity class for our completeness result, namely $\text{succ}\exists\mathbb{R}$ \citep{van2023hardness}. Recall that ETR is the existential fragment of first-order logic, that is, the set of first-order expressions of the form $\exists x_1, \dots,  x_n \;\alpha$, such that $\alpha$ is a quantifier-free. The class $\text{succ}\exists\mathbb{R}$ is defined by the problem of determining whether a circuit computes an encoding of a satisfiable ETR instance:

\begin{definition}[Trees for ETR formulas]
    Represent an ETR formula $\varphi$ as a tree, where the leaf nodes are labeled with the variables of $\varphi$ and the non-leaf nodes are labeled with the operations of $\varphi$, whose children are their operands; we will call this an ETR tree. Fix some encoding on the locations of nodes in the tree, such that the binary string for a node $v$ specifies its location in the tree. We will abuse notation, using $v$ to refer to the encoding of $v$'s location, and using $\varphi$ to refer to the associated ETR tree.
\end{definition}

\begin{definition}\label{def:circuits}
    A Boolean circuit is a directed acyclic graph $G=(V, E)$ such that
    \begin{itemize}
        \item 
        $N$ vertices, \textit{the inputs}, have no predecesors, and $M$ vertices, \textit{the outputs}, have no successors;
        \item 
        each non-input node is labeled with $\land$, $\lor$, or $\neg$, and has the correct number of predecessors in $G$; and
        \item 
        the edges $E$ are ordered, to distinguish different arguments to the same Boolean expression.
    \end{itemize}
    % The assumption of acyclicity can be relaxed, so long as the outputs of the graph are determined entirely by its inputs \cite{riedel2012cyclic}.
\end{definition}

\begin{definition}\label{def:succETR}
    A Boolean circuit $C:\{0,1\}^N \rightarrow \{0,1\}^M$ \textit{computes} $\varphi$ if $C(v) = (\ell_v, \text{parent}_v, \text{children}_v)$ contains the label for $v$ as well as the locations of its parent and children. % Maybe make the encoding more explicit?

    % Without loss of generality, suppose that for some $n \leq N$, the first $n$ binary strings $s_1,...,s_n$ in $\{0,1\}^N$ represent variables in the ETR instance, which are linearly ordered: $C(s_i) = (z_i, \textit{pa}_i)$ for $i \in [n]$. Suppose, in addition, that the root node $v$ of the ETR tree $\varphi$ computed by $C$ is the last possible input string: $v = 1^N$.
\end{definition}

\begin{definition}
    $\text{succETR}$ is the problem of deciding whether a circuit $C: \{0,1\}^N \rightarrow \{0,1\}^M$ computes a satisfiable ETR instance.
\end{definition}

\begin{definition}
    $\text{succ}\exists\mathbb{R}$ is the class of problems that are reducible in polynomial time to $\text{succETR}$.
\end{definition}

\section{Complexity of Satisfiability}\label{section: complexity}

The introduction observed a negative result:
\begin{theorem}\label{theorem: not RE}
$\mathsf{SAT}_{\mathcal{L}_{\text{prob}}}$ and $\mathsf{SAT}_{\mathcal{L}_{\text{causal}}}$ over either $\mathcal{M}_{\text{fin}}$ or $\mathcal{M}$ are undecidable.
\end{theorem}
This section shows the following positive results:
\begin{theorem}\label{theorem: succETR completeness}
    $\mathsf{SAT}_{\mathcal{L}_\text{causal}(N)}$ over $\mathcal{M}_{N}$ is complete for $\text{succ}\exists\mathbb{R}$.
\end{theorem}

\begin{theorem}\label{RE-Pi2}
    $\mathsf{SAT}_{\mathcal{L}_{\text{prob}}}$ and $\mathsf{SAT}_{\mathcal{L}_{\text{causal}}}$ over $\mathcal{M}_{\text{fin}}$ are recursively enumerable. % and over $\mathcal{M}$ is $\Pi_2$.
\end{theorem}

To show Theorem~\ref{theorem: succETR completeness}, we use a characterization of $\text{succ}\exists\mathbb{R}$ by \citealt{blaeser2024existential}, which allows us to lift the results of \cite{mosse2022causal} to the setting with summation operators:

\begin{definition}
    A \textit{real RAM} is a tuple $(P, T, R)$, where $P$ is a program, $T$ is an infinite memory tape, and $R$ is a set of registers. There are two kinds of registers: real registers and word registers, which respectively store real numbers and binary strings. The program $P$ can move values between memory and registers, as well as add, subtract, multiply, and divide values stored in registers; $P$ can also perform bitwise Boolean operations on binary registers. A \textit{nondeterministic real RAM} $M$ receives a binary string as input $I$, together with a certificate $C = (x, s)$, where $x$ is a sequence of real numbers and $s$ is a binary string. Then $M$ accepts $I$ if there exists a certificate $C$ such that $M(I, C) =1$. (For details, see \citealt{erickson2022smoothing}.)

\end{definition}

\begin{theorem}[\citealt{blaeser2024existential}]\label{theorem: nondeterministic real ram}
$\text{succ}\exists\mathbb{R}$ is the precisely the sets of languages decided by nondeterministic real RAMs in exponential time.
\end{theorem}

We include a sketch of the proof by \cite{blaeser2024existential}:

\begin{proof}
    One direction is trivial: an exponential-time nondeterministic real RAM can simply use a circuit to write out the ETR instance it computes and solve this instance. For the other direction, it suffices to show that given a real RAM $M$, an input $I$, and $w,t \in \mathbb{N}$, one can construct a circuit $C$ of size $\text{poly}(|M|\cdot|I|\cdot w)\cdot \text{polylog}(t)$ in time $\text{poly}(|M|\cdot|I|\cdot w)$ encoding an ETR formula $\varphi$ which is satisfiable iff $M$ accepts $I$ within $t$ steps.

    We begin with the special case where $M$ is polynomial-time (Lemma 11 of \citealt{erickson2022smoothing}). In this case, let $\varphi$ be the formula
    \begin{multline*}
       \exists x_0,...,x_n \; \texttt{PowersOf2}(x_0,..,x_n) \land  \texttt{WordsAreWords}(x_0,..,x_n) \land  \texttt{FixInput}(x_0,..,x_n)\\ \land \texttt{ProgramCounter}(x_0,...,x_n) \land \texttt{Execute}(x_0,..,x_n). 
    \end{multline*}
    
    The subformulas are defined as follows:
    \begin{itemize}
        \item 
        \texttt{PowersOf2} ensures that the first $w$ variables $x_0,...,x_{w-1}$ are equal to the first $w$ powers of $2$.
        \item 
        $\texttt{IsWord}(X)$ for $X = x_i,...,x_j$ ensures that $X$ describes a number in binary between $x_0 = 1$ and $x_{w-1} = 2^{w-1}$. The sub-formula $\texttt{WordsAreWords}$ selects a subset $X_1,..,X_k$ of the variables $x_w,...,x_n$, ensuring that $\texttt{IsWord}(X_i)$ for $i \in [k]$. These variables represent the word registers of $M$.
        \item 
        \texttt{FixInput} ensures that some of the word variables encode the input $I$ provided to $M$.
        \item 
        $\texttt{ProgramCounter}(x_0,...,x_n)$ designates a single variable $x_t$ to indicate the time-step of computation.
        \item 
        Using the standard reduction provided by the Cook-Levin Theorem, $\texttt{Execute}(x_0,..,x_n)$ simulates $M$, using and updating the program counter at each time step.
    \end{itemize}

    Since $\varphi$ precisely simulates runs of $M$, there exist a certificate and input on which $M$ accepts iff there exist variables $x_0,...,x_n$ satisfying $\varphi$. \cite{blaeser2024existential} show that in the general case, one can construct a succinct circuit computing this formula in polynomial-time; correctness for the construction follows immediately from correctness in the special case.
\end{proof}

The completeness results of \cite{mosse2022causal} relied on small model property. This is no longer guaranteed, but given the above characterization, an exponential model property will suffice:

% Make \Delta_\varphi contain all possible interventions too. This will make each complete state description exponential in size. But that is fine. It will ensure that \Delta_\varphi = \Delta_\psi.

\begin{definition}
    Fix a formula $\varphi \in  \mathcal{L}_{\mathrm{causal}}$. Let $\mathbf{V}_\varphi$ denote all variables appearing in $\varphi$, and let $\mathcal{I}$ contain all interventions over all values of variables appearing in $\varphi$. Let $\Delta_\varphi$ contain all interventions $\alpha \in \mathcal{I}$ paired with all possible assignments resulting from those interventions:
\begin{align*}
    \Delta_\varphi = \Big\{\bigwedge_{\alpha \in \mathcal{I}} \; \Big( [\alpha] \bigwedge_{V \in \mathbf{V}_\varphi} \beta_V^\alpha\Big) \;  :\; \beta_V^\alpha \in \text{Val}(V) \text{ for }V \in \textbf{V}_{\varphi}\Big\}
\end{align*}
    We note that interventions and values are not restricted to those appearing in $\varphi$; this is a departure from the definition of $\Delta_\varphi$ in \cite{mosse2022causal}.
\end{definition}

% NOTE: Be clearer about sequents and implications
% NOTE: Remember that values are finite now

% NB: It is a little silly to call it an exponential model, since the full set of \Delta is exponential. Really we are just saying that it has a model in some subset of the \Delta.

\begin{lemma}[Exponential model property]\label{lemma: exponential model}
    Fix a formula $\varphi $ in $\mathcal{L}_{\text{causal}}^{\text{closed}}(N)$. If $\varphi$ is satisfiable, then it is satisfied by a model $\mathfrak{M}$ which assigns positive probability to at most $2^{\text{poly}(|\varphi|)}$ complete state descriptions $\delta \in \Delta_{\varphi}$.
\end{lemma}

% Key problem that prevents proof of small model property from extending to this case: the summation can contain a product. (Also, with the naive construction, each delta could be exponentially long, if a sum ranges over values of an intervention.)

\begin{proof}
    Transform $\varphi$ into a formula $\varphi^\prime$ by writing out each sum and range variable or coefficient explicitly. For example, $ \sanssum_{x_i}x_i \cdot \probsymbol(X = x_i) $ is replaced with $1 \cdot \probsymbol\big(X=c_1^X\big)  + 2 \cdot \probsymbol\big(X = c_2^X\big) + \dots + n \cdot \probsymbol\big(X = c_n^X\big)$, where the natural number $n$ is shorthand for $\probsymbol(\top) + \dots + \probsymbol(\top)$, the $n$-term sum of $\probsymbol(\top)$. (Note that this formulation of $\varphi^\prime$ requires the closed-world and finite range assumptions: each value in the finite range of $X$ is named by a constant.)
    
    Then $\varphi^\prime$ contains no summation operators and no unbound assignments; it belongs to the fragment $\mathcal{L}$ of $\mathcal{L}_{\text{causal}}$ without summation or free variables. Let $\Delta_{\varphi^\prime}$ contain all interventions appearing in $\varphi^\prime$ paired with all possible assignments. Note that $\Delta_{\varphi^\prime} \subseteq \Delta_\varphi$. By the small model property for $\mathcal{L}$ (Lemma~4.7 of \citealt{mosse2022causal}), $\varphi^\prime$ has a model which assigns positive probability to at most $|\varphi^\prime| \leq 2^{\text{poly}(|\varphi|)}$ complete state descriptions $\delta \in \Delta_{\varphi^\prime}$. This is also a model of $\varphi$, given that $\Delta_{\varphi^\prime} \subseteq \Delta_\varphi$ and the construction of $\varphi^\prime$ from $\varphi$. %Further, because in this setting the range of each variable is considered a fixed constant, $|\varphi^\prime| \leq 2^{\text{poly}(|\textbf{V}_\varphi|)}$, as desired.
\end{proof}

We now show completeness, using the following hardness result:

\begin{theorem}[\citealt{van2023hardness}]\label{theorem: prob SAT is succETR-complete}
    The satisfiability problem for $\mathcal{L}_{\text{prob}}^{\text{closed}}(N)$ is $\text{succETR}$-hard.
\end{theorem}

\begin{proof}[Proof of Theorem~\ref{theorem: succETR completeness}]
     By Theorem~\ref{theorem: prob SAT is succETR-complete}, satisfiability for $\mathcal{L}_{\text{prob}}^{\text{closed}}(N)$, and thus for $\mathcal{L}_{\text{causal}}(N)$, is succETR-hard. It then suffices to show that satisfiability for $\mathcal{L}_{\text{causal}}(N)$ is in succETR. Consider any sequent $\varphi \Rightarrow \psi$ in $\mathcal{L}_{\text{causal}}$. By Theorem~\ref{theorem: nondeterministic real ram}, it suffices to construct a nondeterministic exponential time real RAM $M$ which decides whether $\varphi \Rightarrow \psi$ is satisfiable:

    \begin{enumerate}
        \item 
        For a possible assignment $\iota$ of free range variables appearing in $\varphi$, let $\varphi_\iota$ denote the result of replacing all free variables with constants denoting their assignments, e.g. replacing $\mathsf{P}(V= v_i)$ with $\mathsf{P}(V = c_j^V)$, where $j = \iota(V, i)$.  Let $\overline{\varphi} = \bigwedge_\iota \varphi_\iota$, and similarly for $\overline{\psi}$. Then $\overline{\varphi}, \overline{\psi} \in \mathcal{L}_{\text{causal}}^{\text{closed}}$, and $\varphi \Rightarrow \psi$ iff $\overline{\varphi} \Rightarrow \overline{\psi}$. Let $M$ construct $\overline{\varphi} \Rightarrow \overline{\psi}$, which is exponential in the size of the input formula. Since this sequent has no free variables, it is equisatisfiable with $\overline{\varphi} \rightarrow \overline{\psi}$ and may be considered a single formula $\varphi^\prime$.
        \item 
        Using Lemma~\ref{lemma: exponential model}, let $M$ receive as a certificate the support $\Delta^\prime$ of an exponential model for $\varphi^\prime$ and an ordering $\prec$ on the variables $\textbf{V}_{\varphi^\prime}$. Note that $|\Delta^\prime| \leq |\Delta_{\varphi^\prime}| \leq 2^{\text{poly}(|\varphi|)} $.
        \item
        $M$ will then check that $\Delta^\prime \subseteq \Delta_{\varphi^\prime}$ and that the $\delta \in \Delta^\prime$ do not induce influence relations which are incompatible with the ordering $\prec$ (Lemma~4.5 of \citealt{mosse2022causal}).\footnote{A model $\mathfrak{M}$ \textit{induces the influence relation} $V_i \rightsquigarrow V_j$ when there exist values $v, v^\prime \in \text{Val}(V_j)$ and interventions $\alpha,\alpha^\prime$ differing only in the value they impose upon $V_i$ for which
$\mathfrak{M} \vDash \probsymbol\big([\alpha] V_j = v \land [\alpha^\prime] V_j = v^\prime\big)  > 0.$ Given an enumeration of variables $V_1,...,V_n$ compatible with a well-order $\prec$, the model $\mathfrak{M}$ is \textit{compatible with} $\prec$ when it induces no instance $V_i \rightsquigarrow V_j$ with $i > j$.}
        \item 
        $M$ will form a formula $\varphi_{\text{prob}}$ from $\varphi^\prime$ by replacing each $\probsymbol(\epsilon)$ in $\varphi$ with
        \[
        \probsymbol\bigg(\bigvee_{\substack{\delta \in \Delta^\prime\\\delta \vDash \epsilon}} f(\delta)\bigg),
        \]where $f$ is a bijection between $\Delta^\prime$ and an arbitrary set of mutually unsatisfiable propositional statements built from the atoms $X = c^X_3, Y= c^Y_1$, and so on.

        \item 
        Now $\varphi_{\text{prob}} \in \mathcal{L}_{\text{prob}}^{\text{closed}}(N)$ and contains no marginalization. Satisfiability for $\mathcal{L}_{\text{prob}}^{\text{closed}}(N)$ without marginalization is ETR-complete \citep{mosse2022causal}. Any problem in ETR is decidable by a nondeterministic, polynomial-time real RAM $M^\prime$ \citep{erickson2022smoothing}, so $M^\prime$ solves the satisfiability problem for $\psi$ in polynomial time. Let $M$ run $M^\prime$ on $\varphi_{\text{prob}}$ and return the result.
    \end{enumerate}

    If $\varphi\Rightarrow \psi$ is satisfiable, then there is some probability assignment to the $\delta \in \Delta^\prime$ which satisfies it (by Lemma~\ref{lemma: exponential model}); this assignment is also a model of $\varphi_{\text{prob}}$. Conversely, if there is a model of $\varphi_{\text{prob}}$, then the mutual unsatisfiability of the $f(\delta)$ ensures the existence of a causal model for $\varphi^\prime = \varphi \rightarrow \psi$, and thus for $\varphi \Rightarrow \psi$ (Lemma 4.6 of \citealt{mosse2022causal}). Thus the map $\varphi \Rightarrow \psi \mapsto \varphi_{\text{prob}}$ preserves and reflects satisfiability, and $M$ decides the satisfiability of the latter.
\end{proof}

% \textcolor{blue}{Can we handle infinite ranges?}

In addition to resolving an open question about the complexity of $\mathsf{SAT}_{\mathcal{L}_{\text{causal}}(N)}$, the above immediately gives an independent proof of the completeness results \cite{van2023hardness} obtained for languages intermediate between $\mathcal{L}_{\text{causal}}(N)$ and $\mathcal{L}_{\text{prob}}^{\text{closed}}(N)$. We note that \cite{dorfler2024probabilistic} independently obtained the above result, by the same argument, for the fragment of $\mathcal{L}_{\text{causal}}$ without coefficients or free variables.

The above result can also be used to show recursive enumerability for $\mathsf{SAT}_{\mathcal{L}_{\text{causal}}}$ over merely finite models:

\begin{comment}

Using $\mathcal{M}$ as an abbreviation for the set of induced probability distributions on $\text{Val}(\mathbf{V})$:
\begin{lemma}
$\mathcal{M}$ is compact
\end{lemma}
\begin{proof}
For each $p: \mathbb{N}\to [0, 1]$, reorder so that $p(1) \geq p(2) \geq \dots$
\end{proof}
\begin{lemma}
$\mathcal{M} = \mathrm{cl} \bigcup_{n=1}^\infty \mathcal{M}_n$
\end{lemma}
\begin{lemma}
$\varpi(\mathcal{M}) = \mathrm{cl}\bigcup_{n=1}^\infty \varpi(\mathcal{M}_n)$
\end{lemma}
\begin{proof}
This follows since $\mathcal{M}$ is compact, $\varpi$ is continuous so $\varpi(\mathcal{M})$ is compact.
Thus $\varpi(\mathcal{M})$ is closed and $\mathrm{cl} \varpi\big(\bigcup_{n=1}^\infty \mathcal{M}_n\big) = \mathrm{cl}\bigcup_{n=1}^\infty \varpi(\mathcal{M}_n)$.
\end{proof}

\end{comment}

\begin{proof}[Proof of Theorem~\ref{RE-Pi2}]
    We show that $\mathsf{SAT}_{\mathcal{L}_{\text{causal}}}$ over $\mathcal{M}_{\text{fin}}$ is recursively enumerable. Fix a satisfiability problem $\Gamma \Rightarrow \psi$; with $\Gamma$ assumed finite, we may replace it with a single formula $\varphi$. Enumerate over all possible finite ranges of all variables appearing in $\varphi,\psi$.

    For each such finite range, by Theorem~\ref{theorem: succETR completeness}, the satisfiability problem for $\varphi \Rightarrow \psi$ is computable by some Turing machine $M$. Run $M$ and declare $\varphi \Rightarrow \psi$ satisfiable if $M$ returns yes. If $\varphi \Rightarrow \psi$ is satisfiable over $\mathcal{M}_\text{fin}$, then it is satisfiable for some choice of finite ranges, and this procedure will correctly declare it satisfiable.%\footnote{}
    % Now, we show that the same problem over $\mathcal{M}$ is $\Pi_2$.
\end{proof}

\section{Axiomatization}\label{section: axiomatization}

We now work toward an axiomatization of the entailment relation $\Gamma \vDash \varphi$ for probabilistic languages.
We will make the coefficient restriction of \eqref{eqn:circterms}, focusing on the languages
$\mathcal{L}^{\ddagger}_{\text{prob}}$ and $\mathcal{L}^{\ddagger}_{\text{prob}}(N)$, where $\ddagger \in \{ \circ, \circ+\text{closed} \}$.
Note that our semantics have possibly unintuitive outcomes when it comes to conditional probability.
For example, even the formula $\probsymbol(Y = y \mid X = x) \succsim \underline{0}$ or $\probsymbol(Y = y \mid X = x) \succ \underline{0} \lor \lnot \probsymbol(Y = y \mid X = x) \succ \underline{0}$ is invalid, as witnessed by any model assigning zero measure to some $X = x$, and just as in Kleene's 3-valued logic there are no Boolean tautologies.
One solution is to ``guard'' each formula with an antecedent saying that all its conditions occur with positive probability.
However, in order to simplify the results, we will instead assume every SCM $\mathfrak{M}$ to be \emph{positive}, meaning that the observational joint distribution $\mathbb{P}_{\mathfrak{M}}(\mathbf{V})$ is strictly positive, i.e., such that the measure of any clopen cylinder set is nonzero. (Recall a cylinder set is any $S = \bigtimes_{V \in \mathbf{V}} S_V$ where $S_V = \mathrm{Val}(V)$ for all except finitely many $V$.) Although this assumption is stronger, it is ubiquitous in causal inference (see, e.g., \citealt{Shpitser,Pearl2009} for typical examples).
\begin{definition}\label{def:posmod}
Let $\mathcal{M}^+, \mathcal{M}_{\text{fin}}^+, \mathcal{M}_N^+$ be the positive subclasses of $\mathcal{M}, \mathcal{M}_{\text{fin}}, \mathcal{M}_N$ respectively.
\end{definition}

% \begin{definition}
%     We say that a measurable model $\mathfrak{M}$ which is recursive via $\prec$ is \textit{positive} if for every system of equations $\mathcal{F}$ consistent with $\prec$,
%     \[
%     \mathbb{P}_{\mathfrak{M}}(\delta_\mathcal{F}) > 0,
%     \]
%     where $\delta_{\mathcal{F}}$ is the complete description of $\mathcal{F}$. For example, if $X \prec Y$ are binary-valued random variables, and $\mathcal{F}$ contains the equations $f_X = 0$ and $f_Y = X$, the event $\delta_{\mathcal{F}}$ is $[ \top ] X = 0 \land [X = 0] Y = 0 \land [X = 1] Y = 1)$.
% \end{definition}

The assumption of positivity ensures that conditional probabilities are well-defined, provided we never condition on events equivalent to $\bot$.
We will therefore restrict conditions to the fragment $\mathcal{L}_{\text{cond}} \subset \mathcal{L}_{\text{base}}^{\mathrm{prob}}$ consisting of $\top$ and pure conjunctions of literals ($V = d$ or $\lnot V = d$) such that for any $V \in \mathbf{V}$, a literal involving $V$ appears at most once in $\varphi$. It is straightforward to see that any $\varphi \in \mathcal{L}_{\text{cond}}$
receives a positive probability under any positive $\mathfrak{M}$ and assignment $\iota$.

\begin{definition}\label{def:poslang}
Let $\mathcal{L}_{\text{prob}^+}^\ddagger$ for $\ddagger \in \{ \circ, \circ+\text{closed} \}$ be the fragment of $\mathcal{L}_{\text{prob}}^\ddagger$ where in any primitive term $\probsymbol(\delta \;|\: \delta')$ from \eqref{eqn:circterms} we require $\delta' \in \mathcal{L}_{\text{cond}}$.
Let $\mathcal{L}_{\text{prob}^+}^\ddagger(N)$ be the same over signatures of uniform cardinality $N$ (cf. \S{}\ref{sec:languages:bounded}).
With $\mathcal{L}$ standing for $\mathcal{L}_{\text{prob}^+}^\circ$, we abbreviate these four as $\mathcal{L}_{**}^*$ for $* \in \{\varepsilon, \text{closed}\}, ** \in \{\varepsilon, N\}$.
\end{definition}

% , over $\mathcal{M}^+$ and $\mathcal{M}_{\text{fin}}^+$.
% Using a set of cardinality $N$ for $\mathcal{C}_V$ as in \S{}\ref{sec:languages:bounded}, the same schemata work for $\mathcal{L}^{\ddagger}_{\text{prob}}(N)$ over $\mathcal{M}^+(N)$ for each $N \ge 1$.
% We first introduce the system for $\mathcal{L}^{\circ + \text{closed}}_{\text{prob}}$ in  before proceeding to augment it for the wider language $\mathcal{L}^{\circ}_{\text{prob}}$ that includes free variables in .

\subsection{Axioms}\label{sec:comp:ax}
We start with the system $\mathsf{AX}_{\text{poly}}$ from \citet{IBELING2023103339}. This includes Boolean tautologies across probability comparisons as well as equalities between variable symbols.
As for new axioms, first, we include properties of $\equiv$; for any $c, c' \in \mathcal{C}_V$ and $V \in \mathbf{V}$:
\begin{eqnarray*}
\mathsf{EqReflex}. && c \equiv c\\
\mathsf{EqReplace}. && c \equiv c' \rightarrow( \varphi[v / c] \rightarrow \varphi[v / c']).
% \mathsf{EqSym}. && d_i \equiv d_j \rightarrow d_j \equiv d_i\\
% \mathsf{EqTrans}. && d_i \equiv d_j \land d_j \equiv d_k \rightarrow d_i \equiv d_k.
\end{eqnarray*}
Another principle states that unequal constants cannot occur simultaneously:
\begin{eqnarray*}
\mathsf{EqDist}. && c \not\equiv c' \rightarrow \probsymbol(V = c \land V = c') \approx \underline{0}.
\end{eqnarray*}
% Axiom for equal constants:
% \begin{eqnarray*}
% d_i \equiv d_j \rightarrow \mathsf{t}[X = x / X = d_i] \approx \mathsf{t}[X = x / X = d_j]
% \end{eqnarray*}
Then, a single axiom schema allows us to reduce conditional to unconditional probability:
\begin{eqnarray*}
\mathsf{Cond}. && \probsymbol(\delta \;|\; \delta') \succsim \mathsf{t} \leftrightarrow \probsymbol(\delta \land \delta') \succsim \mathsf{t}\cdot \probsymbol(\delta').
\end{eqnarray*}
Next, we have a lower bound on sums. Where $S \subset \mathcal{C}_V$ is any finite subset of constants:
\begin{eqnarray*}
\mathsf{SumLower}. && \bigwedge_{\substack{i \neq j\\c^V_i, c^V_j \in S}} c^V_{i} \not\equiv c^V_j \rightarrow \sanssum_{v_i} \mathsf{t} \succsim \sum_{c^V_i \in S} \mathsf{t}[V = {v_i} / V = c^V_i]
\end{eqnarray*}
After that, an axiom captures the strict positivity assumed in our model classes. Where $\alpha$ ranges over $\mathcal{L}_{\mathrm{cond}}$:
\begin{eqnarray*}
\mathsf{Pos}. && \probsymbol(\alpha) \succ \underline{0}.
\end{eqnarray*}
Finally, for the case of finite signatures only, to capture the distinctness requirement in $\mathcal{M}_N$ we introduce the axiom schema
\begin{eqnarray*}
\mathsf{Fin}_N. && \sanssum_{v_i} \probsymbol(\top) \approx \underline{N}. %\bigwedge_{c_i^V, c_j^V \in \mathcal{C}_V} c_i^V \not\equiv c_j^V.
\end{eqnarray*}

\subsection{Rules}\label{sec:comp:rules}
The system $\mathsf{AX}_{\text{poly}}$ we adopted from prior work has only one rule, \emph{modus ponens}, i.e., $\{ \varphi, \varphi \rightarrow \psi\} \vdash \psi$ or
\begin{align*}
\infer[\mathsf{MP}.]{\psi}{\varphi, \varphi\rightarrow \psi}
\end{align*}
As we will illustrate by several different examples, all our languages are incompact over their respective model classes, necessitating the use of rules with infinitely many premises in order to obtain strongly complete axiomatizations.
% Consider the set $\{  \}$
We first introduce a variant of the infinitary rule R3 of \citet{Perovic}:
\begin{align*}
\infer[\mathsf{Conv}]{\varphi \rightarrow \mathsf{t} \precsim \underline{0}}{\varphi \rightarrow \mathsf{t} \precsim \underline{1/n} \text{ for every } n = 1, 2, 3, \dots}
% \mathsf{Conv}. && \big\{  \big\} \vdash .
\end{align*}
which allows eliciting a contradiction from finitely-satisfiable, unsatisfiable sets, such as $$\{\probsymbol(V = c_1^V) \precsim \underline{1/n} : n = 1, 2, 3, \dots \} \cup \{\probsymbol(V = c_1^V) \succ \underline{0} \}.$$
%Consider the set $\{\probsymbol(X = c_X^n) \approx \underline{0} : n = 1, 2, 3, \dots \}$.
%Clearly the set is unsatisfiable; however it is finitely satisfiable in either $\mathcal{M}$ or $\mathcal{M}_{\text{fin}}$.
%For each $V \in \mathbf{V}$ enumerate the set of $V$-constants $\mathcal{C}_V \cup \mathcal{C}'_{V}$ as $\{ c_V^1, c_V^2, c_V^3, \dots \}$.

Next, we add the following rule (or collection thereof), for every finite $\mathbf{X} \subset \mathbf{V}$ and rational $q > 0$:
\begin{align*}
%\Bigg\{
\infer[\mathsf{Unity}]{\varphi \rightarrow \underline{1} - \probsymbol\Big(\bigwedge_{X \in \mathbf{X}}  X = c_X^1 \lor \lnot X = c_X^1\Big) \succ \underline{q}}{\varphi \rightarrow \underline{1} - \probsymbol\Big(\bigwedge_{X \in \mathbf{X}} \bigvee_{\substack{1\le i \le n\\ c^X_i \in \mathcal{C}_X}} X = c_X^i\Big) \succ \underline{q} \text{ for every } n = 1, 2, 3, \dots}
%\Bigg\} \vdash
\end{align*}
Note that from its conclusion we can deduce $\varphi \rightarrow \bot$; nevertheless, it will be important for our proof strategy that it have the specific form above.
Also observe that the rule above is infinitary iff we have infinitely many constant symbols; the set of formulas $\big\{ \probsymbol(V = c_V^n) \approx \underline{1/3^n} : n = 1, 2, 3, \dots \big\}$ in such a signature is finitely satisfiable but unsatisfiable, and $\mathsf{Unity}$ derives a contradiction from it.
% Let $\mathsf{Unity}(N)$ be the restricted version of $\mathsf{Unity}$ to ranges of size $N$, i.e., the same rule but with a subset of premises formed by intersecting the original with $\mathcal{L}(N)$. Thus, letting $\psi_n = \varphi \rightarrow \underline{1} - \probsymbol\Big(\bigwedge_{X \in \mathbf{X}} \bigvee_{1\le i \le n} X = c_X^i\Big) \succ \underline{q}$, the premises of $\mathsf{Unity}(N)$ are $\{ \psi_n : 1 \le n \le N \}$.
% For any $N$, the rule $\mathsf{Unity}(N)$ is finitary, and it is admissible to replace it by an implication axiom in $\mathcal{L}(N)$.
% By enforcing the convergence of joint probabilities to $1$ as more and more constants from the extended language are included, this rule addresses 

% Need: If G ^ ~phi entails bottom, then Gamma entails phi

% Counterexample to deduction:  P(X = x) > 0 -> P(X = c) > 0
% for all x [ P(x) > 0 -> P(c) > 0 ] 
% [ for all x P(x) > 0 ] -> P(c) > 0 VALID

After that, the following rule provides upper bounds on sums.
Where $\varphi$ is a formula and $\mathsf{t}$ is some term, $n \ge 1$ and $\Pi = \{ \Pi_\upsilon\}_{\upsilon \in \Upsilon}$ is a partition of $\mathcal{C}_X(n) = \{c^X_i \in \mathcal{C}_X : 1\le i \le n \}$, let $\varphi_{n, \Pi}$ denote the formula
\begin{eqnarray}\label{eqn:sumupperinstance}
\varphi_{n, \Pi} \quad := \quad
\varphi \rightarrow \Big[\Big(\bigwedge_{\substack{\upsilon \in \Upsilon \\ c^X_i, c^X_j \in \Pi_\upsilon}} c^X_i = c^X_j \land \bigwedge_{\substack{\upsilon \neq \upsilon' \in \Upsilon \\ c^X_i \in \Pi_\upsilon\\ c^X_j \in \Pi_{\upsilon'}}} c^X_i \neq c^X_j\Big) \rightarrow
\sum_{\upsilon \in \Upsilon} \mathsf{t}\big[X = x_i / \bigvee_{c^X_i \in \Pi_\upsilon} X = c^X_i\big] \precsim 
\mathsf{t}' % \underline{q}
\Big].
\end{eqnarray}
Letting $\mathbf{B}_n$ be the set of partitions of $\mathcal{C}_X(n)$, we have a rule:
\begin{align*}
\infer[\mathsf{SumUpper}.]{\varphi \rightarrow \sanssum_{x_i} \mathsf{t} \precsim \mathsf{t}'}{\varphi_{n, \Pi} \text{ for each } n \in \mathbb{Z}^+, \Pi \in \mathbf{B}_n}
\end{align*}
As with $\mathsf{Unity}$, the rule $\mathsf{SumUpper}$ becomes infinitary with infinite signatures.
A finitely satisfiable set it proves inconsistent is
$\big\{ \bigwedge_{1 \le i \neq j \le n} c^X_i \neq c^X_j \land \sum_{i = 1}^n \probsymbol(X = c^X_i) \precsim \underline{1/2} : n = 1, 2, 3, \dots \big\}$.

Finally, consider the set $\{ \probsymbol( X = c_X^n ) \approx \underline{1/2^n}: n = 1, 2, 3, \dots\}$, satisfiable in $\mathcal{M}$.
In $\mathcal{M}_{\text{fin}}$, this set is finitely satisfiable but unsatisfiable.
The next rule addresses the distinction; as with $\mathsf{Unity}$, its conclusion proves $\varphi \rightarrow \bot$:
\begin{align*}
% && 
%\big\{
\infer[\mathsf{Fin}.]{\varphi \rightarrow \sanssum_{x_i} \probsymbol(\top) \prec \underline{0}}{\varphi \rightarrow \sanssum_{x_i} \probsymbol(\top) \succsim \underline{n} \text{ for every } n = 1, 2, 3, \dots}
\end{align*}
Although $\mathsf{Fin}$ is still an infinitary rule if we have finitely many constants, we will not need it in this case.
% \tb{Do we need an axiom to prevent the sum above being 0? (don't think so)}

% Not needed; see rule used in section about sum of sums
% Finally, the next rule captures not being able to distinguish ``degrees of divergence'' within the extended real line:
% \begin{align*}
% \infer[\mathsf{InfEq}.]{\varphi \rightarrow \mathsf{t}_1 \approx \mathsf{t}_2}{\varphi \rightarrow (\mathsf{t}_1 \succsim \underline{n} \land \mathsf{t}_2 \succsim \underline{n})\text{ for every } n = 1, 2, 3, \dots}
% \end{align*}
% A corresponding set exhibiting incompactness is $\{ \sanssum_x \probsymbol(\top) \not\approx \sanssum_y \probsymbol(\top) \land \sanssum_x \probsymbol(\top) \succsim \underline{n} \land \sanssum_y \probsymbol(\top) \succsim \underline{n} : n = 1, 2, 3, \dots \}$.
% Like $\mathsf{Fin}$, this rule will not be used in case we only have finitely many constants.

Having completed our description of the axioms and rules, we can now define the systems for the closed fragments:
\begin{definition}\label{def:axclosed}
Let $\mathsf{AX}^{\text{closed}} = \mathsf{AX}_{\text{poly}} + \mathsf{EqReflex} + \mathsf{EqReplace} + \mathsf{EqDist} + \mathsf{Cond} + \mathsf{SumLower} + \mathsf{Pos} + \mathsf{Conv} + \mathsf{Unity} + \mathsf{SumUpper}$,
%\begin{align*}
$\mathsf{AX}^{\text{closed}}_N = \mathsf{AX}^{\text{closed}}+ \mathsf{Fin}_N$, and
$\mathsf{AX}^{\text{closed}}_{\text{fin}} = \mathsf{AX}^{\text{closed}} + \mathsf{Fin}$.
%\mathsf{AX}^{\text{closed}} = \mathsf{AX}_{\text{common}}.
%\end{align*}
\end{definition}

% Where $N \ge 1$ denotes some restricted range, let $\mathsf{SumUpper}(N)$ be formed from $\mathsf{SumUpper}$ in the same manner as $\mathsf{Unity}(N)$ from $\mathsf{Unity}$; again, $\mathsf{SumUpper}(N)$ is finitary and in $\mathcal{L}(N)$.

\subsection{Deduction with free variables}
Finally add two rules toward axiomatizing the languages with free variables (implicitly universally quantified):
\begin{eqnarray*}
\infer[\mathsf{FreeElim}]{\varphi[v / c]}{\varphi} \qquad\qquad
\infer[\mathsf{FreeIntro}]{\varphi}{\varphi[v/ c] \text{ for each } c \in \mathcal{C}_V}
\end{eqnarray*}
where $v \in \mathcal{V}_V$ is any range variable for $V$ and in the conclusion of $\mathsf{FreeElim}$, $c \in \mathcal{C}_V$ is any $V$-constant.
As the deduction theorem fails in this setting (see Rmk.~\ref{rmk:deductionfail}), unlike with the previous infinitary rules we have no intuitive example of incompactness associated with $\mathsf{FreeIntro}$.
\begin{definition}\label{def:axopen}
Let $\mathsf{AX}_* = \mathsf{AX}_*^{\text{closed}} + \mathsf{FreeElim} + \mathsf{FreeIntro}$ for each $* \in \{ \varepsilon, N, \text{fin}\}$.
\end{definition}
Although these systems just amount to handling free variables via (infinitely many) constant substitutions, note that those of \S{}\ref{sec:comp:ax}--\ref{sec:comp:rules} at least generally remain sound when schemata range over free variables.

\subsection{Examples}

Before proving completeness, we give a few examples to illustrate reasoning with the systems above.
First, it is straightforward to see that the deduction theorem \emph{does} hold in the closed systems of Definition~\ref{def:axclosed}:
\begin{lemma}[Deduction Theorem]%{{\citet[Theorem~1]{Perovic}}}]
\label{lem:deduction} Let $T \subset \mathcal{L}_*^{\text{closed}}$ and $\varphi, \psi \in \mathcal{L}_*^{\text{closed}}$. Then $T \vdash \varphi \rightarrow \psi$ iff $T \cup \{\varphi\} \vdash \psi$, where $\vdash$ is derivability in ${\mathsf{AX}^{\text{closed}}_*}$. \qed
\end{lemma}

\subsubsection{The bounded case}
Consider the following axioms:
\begin{eqnarray}
\mathsf{Distinct}_N. && \bigwedge_{1 \le i \neq j \le N} c^V_i \not\equiv c^V_j, \\
\mathsf{SumEquals}_N. && \sanssum_{x_i} \mathsf{t} \approx \sum_{i = 1}^N \mathsf{t}[X = {x_i} / X = c_i^X].\label{eqn:ax:sumeq}
\end{eqnarray}
\begin{proposition}
These are interderivable: (1) $\mathsf{AX}_N$; (2) $\mathsf{AX} + \mathsf{Distinct}_N$; (3) $\mathsf{AX} + \mathsf{SumEquals}_N$.
\end{proposition}
\begin{proof}
$(1) \vdash (2)$: if $c_i^V \equiv c_j^V$ for some $i, j$, then $\sanssum_x \probsymbol(\top) \precsim \underline{N-1}$ by $\mathsf{SumUpper}$, contradicting $\mathsf{Fin}_N$;
$(2) \vdash (3)$: derive $\mathsf{SumEquals}_N$ using $\mathsf{SumLower}$ and $\mathsf{SumUpper}$ together;
$(3) \vdash (1)$: trivial by taking $\mathsf{t} = \probsymbol(\top)$.
\end{proof}
\subsubsection{Sums of sums}
We derive the following principle:
\begin{align*}
\sanssum_{v_i} (\mathsf{t}_1 + \mathsf{t}_2) \approx \sanssum_{v_i} \mathsf{t}_1 + \sanssum_{v_i} \mathsf{t}_2.
\end{align*}
Casing on the different possibilities of whether $c^X_i \equiv c^X_j$ for each $X$ and $i, j \le n$, using $\mathsf{SumUpper}$ and $\mathsf{SumLower}$, along with $\mathsf{EqReflex}$ and $\mathsf{EqReplace}$ for reasoning about equalities, we can derive for any $\mathsf{t}'$ that $\sanssum_{v_i} (\mathsf{t}_1 + \mathsf{t}_2) \precsim \mathsf{t}' \leftrightarrow \sanssum_{v_i} \mathsf{t}_1 + \sanssum_{v_i} \mathsf{t}_2 \precsim \mathsf{t}'$. Thus it suffices to show the following rule admissible:
\begin{align*}
\infer{\mathsf{t} \approx \mathsf{t}''}{\mathsf{t} \precsim \mathsf{t}' \leftrightarrow \mathsf{t}'' \precsim \mathsf{t}' \text{ for any term }\mathsf{t}'}.
\end{align*}
To see this, use Lemma~\ref{lem:deduction}.
If $\Gamma$ contains all the premises but $\Gamma \cup \{\mathsf{t} \not\approx \mathsf{t}'' \}$ is consistent, then supposing without loss that $\mathsf{t} \prec \mathsf{t}''$, we can prove $\mathsf{t} \prec \frac{\mathsf{t} + \mathsf{t}''}{\underline{2}} \prec \mathsf{t}''$, a contradiction.

\subsubsection{Conditional independence implication}
Suppose we have some conditional independence implication \eqref{eqn:liCI} that is true over finite ranges, encoded with free variables as $\{\varphi_1, \dots, \varphi_\ell \} \vDash \varphi_0$, where without loss we take those free ones to be $x_1, \dots, x_k$; we want to show $\{ \varphi_1, \dots, \varphi_\ell \} \vdash \varphi_0$ over $\mathsf{AX}_{\text{fin}}$.
Letting $\mathbf{X} = \{X_1, \dots, X_k\}$, note that, by the completeness of $\mathsf{AX}_{\text{poly}}$ for semi-algebraic reasoning \citep{IBELING2023103339}, we can show for each $n$ that $\{\varphi^n_1, \dots, \varphi^n_\ell \} \cup \big\{ \probsymbol\big( \bigwedge_{X \in \mathbf{X}} \bigvee_{i=1}^n X = c^X_i  \big) \approx \underline{1} \big\} \vdash \varphi^n_0$
where $\varphi^n_i = \bigwedge_{1 \le i_1, \dots, i_k \le n} \varphi_i\big[x_1 / c_{i_1}^{X_1} \big] \dots \big[ x_k / c_{i_k}^{X_k}\big]$.
By Lemma~\ref{lem:deduction} for $\mathcal{L}^{\text{closed}}$ and $\mathsf{FreeElim}$ this means $\{ \varphi_1, \dots, \varphi_\ell\} \vdash \probsymbol\big( \bigwedge_{X \in \mathbf{X}} \bigvee_{i=1}^n X = c^X_i  \big) \approx \underline{1}\rightarrow \varphi^m_0$, for each $n$ and $m \le n$.

Since $\{\varphi^n_0 : n \ge 1\}\vdash \varphi_0$ by $\mathsf{FreeIntro}$,
to complete the argument, it suffices to show the following rule is admissible:
\begin{align*}
\infer{\varphi}{\probsymbol\big(\bigwedge_{X \in \mathbf{X}} \bigvee_{i=1}^n X = c^X_i\big) \approx \underline{1} \rightarrow \varphi \text{ for each }n \in \mathbb{Z}^+}.
\end{align*}
To see this, if $\lnot \varphi$ is consistent with the premises, then conclude $\probsymbol\big(\bigwedge_{X \in \mathbf{X}} \bigvee_{i=1}^n X = c^X_i\big) \prec \underline{1}$ for all $n$.
Casing on whether $c^X_i \equiv c^X_j$ for each $X$ and $i, j \le n$, as $n \to \infty$ by consistency with $\mathsf{Unity}$ there must be infinitely many equivalence classes of constants under this relation, for at least some $X$. But this implies inconsistency with $\mathsf{Fin}$ via $\mathsf{SumUpper}$.

\subsection{Completeness}
% Define a basic deductive system $\mathsf{AX}_{\text{sum}} = \mathsf{AX} + \mathsf{Conv} + $ and the following systems out of it:
% \begin{align*}
% \mathsf{AX}^{\text{closed}} = \mathsf{AX}_{\text{sum}} + \mathsf{Unity},\quad
% \mathsf{AX}^{\text{closed}}_{\text{fin}} = \mathsf{AX}_{\text{sum}} + \mathsf{Fin},\quad
% \mathsf{AX}^{\text{closed}}_n = \mathsf{AX}_{\text{sum}} + \mathsf{Fin}_n.
% % \mathsf{AX}_{\text{gen, fin}}. && \mathsf{AX}_{\text{closed, fin}} + \mathsf{AX}_{\text{gen}}.
% \end{align*}
% Where $\mathsf{AX}^{\text{closed}}_*$ represents any of the three systems above also let $\mathsf{AX}_{*} = \mathsf{AX}_* + \mathsf{FreeElim} + \mathsf{FreeIntro}$.
\begin{theorem}\label{thm:axcomplete}
% $\mathsf{Poly} + \mathsf{Conv} + \mathsf{Sum} + \mathsf{Squeeze}$ is sound and strongly complete (in $\mathcal{L}'$) for $\mathcal{L}$.
The six systems above are sound and strongly complete in the respective settings (Table~\ref{tab:axsys}).
\begin{table*}\centering
\begin{tabular}{@{}lll@{}}\toprule
Axiomatic System & Language & Model Class\\
\midrule
$\mathsf{AX}^{\text{closed}}$ & $\mathcal{L}^{\text{closed}}$ & $\mathcal{M}^+$ \\
$\mathsf{AX}^{\text{closed}}_{\text{fin}}$ & $\mathcal{L}^{\text{closed}}$ & $\mathcal{M}^+_{\text{fin}}$ \\
$\mathsf{AX}^{\text{closed}}_{N}$ & $\mathcal{L}_N^{\text{closed}}$ & $\mathcal{M}^+_{N}$ \\
$\mathsf{AX}$ & $\mathcal{L}$ & $\mathcal{M}^+$ \\
$\mathsf{AX}_{\text{fin}}$ & $\mathcal{L}$ & $\mathcal{M}^+_{\text{fin}}$ \\
$\mathsf{AX}_{N}$ & $\mathcal{L}_N$ & $\mathcal{M}^+_{N}$ \\
\bottomrule
\end{tabular}
\caption{Deductive systems (Definitions \ref{def:axclosed}, \ref{def:axopen}), languages (Definition \ref{def:poslang}), and model classes (Definition \ref{def:posmod}) with respect to which we show soundness and completeness.}
\label{tab:axsys}
\end{table*}
\end{theorem}
\begin{proof}
Soundness is straightforward.
As for completeness, it suffices to prove only the closed cases:
\begin{lemma}
Suppose $\mathsf{AX}^{\text{closed}}_*$ is complete for $\mathcal{L}^{\text{closed}}_*$ over $\mathcal{M}^+_*$, where $*$ ranges over the cases in Table~\ref{tab:axsys}.
Then $\mathsf{AX}_*$ is complete for $\mathcal{L}_*$ over $\mathcal{M}^+_*$.
\end{lemma}
\begin{proof}
Suppose $\Gamma \vDash \varphi$ over $\mathcal{M}^+_*$, where $\varphi \in \mathcal{L}_*$.
Write $\varphi/\psi$ for $\psi \in \mathcal{L}^{\text{closed}}_*$ if $\psi = \varphi[x_1 / c^{X_1}_{i_1}]\dots[x_n / c^{X_n}_{i_n}]$ for some random variables $X_1, \dots, X_n \in \mathbf{V}$, corresponding free variables $x_1, \dots, x_n$, and indices $i_1, \dots, i_n \in \mathbb{Z}^+$.
Define $\Delta = \{ \delta \in  \mathcal{L}^{\text{closed}}_* : \gamma/\delta \text{ for some } \gamma \in \Gamma \}$ and $\Psi = \{ \psi \in \mathcal{L}_*^{\text{closed}} : \varphi/\psi \}$.
Then $\Gamma \vDash \varphi$ implies that $\Delta \vDash \psi$ for any $\psi \in \Psi$.
Note that $\Gamma \vdash \delta$ for every $\delta \in \Delta$ using $\mathsf{FreeElim}$, while $\Delta \vdash \psi$ for every $\psi \in \Psi$ by completeness of $\mathsf{AX}_*^{\text{closed}}$ for $\mathcal{L}_*^{\text{closed}}$; finally $\Psi \vdash \varphi$ using $\mathsf{FreeIntro}$.
Thus $\Gamma \vdash \varphi$.
% Then $\Gamma' \vDash \varphi[X = x / X = i_X]$ for any $X, i_X$. By a simple induction, it suffices to prove the case in which $\varphi[X = x / X = i_X] \in \mathcal{L}^{\text{closed}}$.
% Since $\mathsf{AX}_*$ strictly extends $\mathsf{AX}_*^{\text{closed}}$, we have that $\Gamma \vDash \varphi[X = x / X = i_X]$
\end{proof}
As for $\mathsf{AX}^{\text{closed}}_{N}$, the axiom \eqref{eqn:ax:sumeq} is derivable and
allows the elimination of all primitive sum operators, thereby reducing the problem essentially to the setting of \citet{Perovic} where the result has already been shown.
Thus we prove only $\mathsf{AX}^{\text{closed}}$ and the finite case $\mathsf{AX}^{\text{closed}}_{\text{fin}}$.
By Lemma~\ref{lem:deduction} we can prove completeness in the usual way by showing that any consistent set $\Gamma$ is satisfiable, toward which we will now show $\Gamma$ can be extended to some maximal consistent $\Gamma^*$.
The preliminary results below are variations of \citet[Lemma 1, Definition 6, Theorem 2]{Perovic}; we omit their completely analogous proofs.
\begin{lemma}\label{lem:infruleunwind}
Consider any (instance of) one of our infinitary rules, with premises $\varphi \rightarrow \psi_i$ indexed $i \in \mathbb{Z}^+$ and conclusion $\varphi \rightarrow \psi$.
Suppose $T \subset \mathcal{L}$ is consistent but $T \cup \{ \varphi \rightarrow \psi \}$ is inconsistent. Then there is some $n \in \mathbb{Z}^+$ such that $T \cup \{ \varphi \rightarrow \lnot \psi_n\}$ is consistent.
\end{lemma}
% \begin{lemma}%[Cf. {\citealt[Lemma~1]{Perovic}}]
% \label{lem:proofconvunwind}
% Suppose $T \subset \mathcal{L}$ is consistent but $T \cup \{\varphi \rightarrow \mathsf{t} \precsim \underline{0}\}$ is inconsistent.
% Then there is a positive $n$ such that $T \cup \{ \varphi \rightarrow \mathsf{t} \succ \underline{1/n}\}$ is consistent. \qed
% \end{lemma}
% \begin{lemma}\label{lem:proofunityunwind}
% Suppose $T \subset \mathcal{L}$ is consistent but $T \cup \Big\{\varphi \rightarrow \underline{1} - \probsymbol\Big(\bigwedge_{X \in \mathbf{X}}  X = c_X^1 \lor \lnot X = c_X^1\Big) \succ \underline{q} \Big\}$ is inconsistent.
% Then there is $n$ such that $T \cup \{ \varphi \rightarrow \underline{1} - \probsymbol\Big(\bigwedge_{X \in \mathbf{X}} \bigvee_{1\le i \le n} X = c_X^i\Big) \precsim \underline{q} \}$ is consistent. \qed
% \end{lemma}
% \begin{lemma}\label{lem:prooffinunwind}
% Suppose $T \subset \mathcal{L}$ is consistent but $T \cup \Big\{\varphi \rightarrow \probsymbol(X = c^1_X) \prec \underline{0} \Big\}$ is inconsistent.
% Then there is $n$ such that $T \cup \{ \varphi \rightarrow \probsymbol(X = c_X^n) \precsim \underline{0} \}$ is consistent. \qed
% \end{lemma}
% \begin{proof}
% Like \citet[Lemma~1]{Perovic} but using $\mathsf{Unity}$.
% \end{proof}
\begin{definition}%[Cf. {\citealt[Definition~6]{Perovic}}]
\label{def:completion}
Enumerate the formulas of $\mathcal{L}$ as $\{ \varphi_1, \varphi_2, \varphi_3, \dots\}$.
Let $T \subset \mathcal{L}$ be consistent, and define for each $i$ a set $T_i$ as follows:
\begin{enumerate}
\item $T_0 = T$.
\item If $T_i \cup \{ \varphi_i \}$ is consistent then $T_{i+1} = T_i \cup \{ \varphi_i \}$.
\item If $T_i \cup \{ \varphi_i \}$ is inconsistent then:
\begin{enumerate}
\item If $\varphi_i$ is $\varphi \rightarrow \mathsf{t} \precsim \underline{0}$ then let $T_{i+1} = T_i \cup \{ \varphi \rightarrow \mathsf{t} \succ \underline{1/n}\}$ where $n$ is that guaranteed by Lemma~\ref{lem:infruleunwind}, applied to $\mathsf{Conv}$, such that $T_{i+1}$ is consistent.\label{firstcasecompletion}
\item If $\varphi_i$ is $\varphi \rightarrow \Big[\underline{1} - \probsymbol\Big(\bigwedge_{X \in \mathbf{X}}  X = c_X^1 \lor \lnot X = c_X^1\Big) \succ \underline{q} \Big]$ then let $T_{i+1} = T_i \cup \{ \varphi \rightarrow \Big[\underline{1} - \probsymbol\Big(\bigwedge_{X \in \mathbf{X}} \bigvee_{1\le i \le n} X = c_X^i\Big) \precsim \underline{q} \Big]\}$ where $n$ is that guaranteed by Lemma~\ref{lem:infruleunwind}, applied to $\mathsf{Unity}$, such that $T_{i+1}$ is consistent.
\item If $\varphi_i$ is $\varphi \rightarrow \sanssum_x \mathsf{t} \precsim \mathsf{t}'$ then let $T_{i+1} = T_i \cup \{ \varphi \rightarrow \lnot\varphi_{n, \Pi}\}$ where $\varphi_{n, \Pi}$ is the formula \eqref{eqn:sumupperinstance}  and $n, \Pi$ are guaranteed by Lemma~\ref{lem:infruleunwind}, applied to $\mathsf{SumUpper}$, to make $T_{i+1}$ consistent.
\item In the finite case: if $\varphi_i$ is $\varphi \rightarrow \sanssum_x \probsymbol(\top) \prec \underline{0}$ then let $T_{i+1} = T_i \cup \{ \varphi\rightarrow \sanssum_x \probsymbol(\top) \prec \underline{n} \}$ where $n$ is guaranteed by Lemma~\ref{lem:infruleunwind}, applied to $\mathsf{Fin}$, to make $T_{i+1}$ consistent.\label{lastcasecompletion}
\item Otherwise let $T_{i+1} = T_i$.
\end{enumerate}
\end{enumerate}
Then let $T^* = \bigcup_{i = 1}^\infty T_i$.
The above are disjoint and exhaustive since the structural form of $\varphi_i$ is distinct for all cases \ref{firstcasecompletion}--\ref{lastcasecompletion} and any possible instances of rules therein.
\end{definition}
\begin{lemma}[Lindenbaum's]
Suppose $T$ is consistent with $T^*$  that defined by Definition~\ref{def:completion}.
Then (1) $T^* \vdash \varphi$ implies $\varphi \in T^*$,
(2) not every formula is in $T^*$, and
(3) for each $\varphi$ either $\varphi \in T^*$ or $\lnot \varphi \in T^*$. \qed
\end{lemma}
Thus we extend $\Gamma$ to $\Gamma^*$; we will now define a model $\mathfrak{M}^*$ satisfying $\Gamma^*$.
We first work toward defining the constant interpretations and ranges. We have:
\begin{lemma}
Define a relation $\sim$ on $\mathcal{C}_V$ by $c_i^V \sim c_j^V$ iff $c_i^V \equiv c_j^V \in \Gamma^*$. Then $\sim$ is an equivalence relation.
\end{lemma}
\begin{proof}
By $\mathsf{EqReplace}$ with $\varphi = v \equiv c$ we have $c\equiv c' \rightarrow (c \equiv c \rightarrow c' \equiv c)$ so we can derive $c \equiv c' \rightarrow c' \equiv c$ since $\mathsf{EqReflex}$ is an axiom.
Transitivity is clear as well, since we have $c' \equiv c'' \rightarrow ( c \equiv c' \rightarrow c \equiv c'' )$.
\end{proof}
For each $V \in \mathbf{V}$ the set of equivalence classes under $\sim$ is countable; if finite of cardinality $M$, define $\Val(V) = \{1, \dots, M\}$ and if infinite, define $\Val(V) = \mathbb{Z}^+$, with $g_V : \Val(V) \to \mathcal{C}_V$ a map defining an association of elements $v$ of $\Val(V)$ to equivalence classes $[g_V(v)]$ represented by $g_V(v)$, i.e. such that $[g_V(v)] \neq [g_V(v')]$ for $v \neq v' \in \Val(V)$.
In particular $g_V$ is injective and 
we define the constant interpretations $(c^V_i)^{\mathfrak{M}^*} = g_V^{-1}(c^V_i)$.
In the finite cases, consistency with $\mathsf{Fin}$ implies that each $\Val(V)$ is finite, since otherwise by $\mathsf{SumLower}$ we can show $\sanssum_v \probsymbol(\top) \succsim \underline{n}$ for each $n$.
We now define a measure $\mu$ on a basis of cylinder sets to obtain the model $\mathfrak{M}^*$.
That is, we define the measure $\mu(S)$ on sets $S = \bigtimes_{V \in \mathbf{V}} S_V$ where $S_V = \mathrm{Val}(V)$ for all except finitely many $V$. Equivalently, assuming $S$ is nonempty, we have $S = \pi^{-1}_{\mathbf{X}}\big((i_X)_{X \in \mathbf{X}}\big)$ for some finite $\mathbf{X} \subset \mathbf{V}$ and values $i_X \in\Val(X)$ for each $X \in \mathbf{X}$, where $\pi_{\mathbf{X}} : \bigtimes_{V \in \mathbf{V}}\Val(V) \to \bigtimes_{X \in \mathbf{X}}\Val(X)$ is the projection map.
We define:
\begin{align}
\mu\big[\pi^{-1}_{\mathbf{X}}\big((i_X)_{X \in \mathbf{X}}\big) \big] = \mathrm{sup} \Big\{s \in [0, 1] \cap \mathbb{Q} : \Gamma^* \vdash \probsymbol\Big(\bigwedge_{X \in \mathbf{X}} X = g_X(i_X) \Big)\succsim \underline{s}\Big\}.\label{eqn:defmu}
\end{align}
\begin{lemma}[{cf. \citealt[Lemma~2]{Perovic}}]
$\mu$ extends uniquely to a positive probability measure. \qed
\end{lemma}
\begin{proof}
From standard extension theorems, with normalization guaranteed by consistency with $\mathsf{Unity}$ and positivity by consistency with $\mathsf{Pos}$.
\end{proof}
% \begin{proof}
% We prove that $\mu$ is a 
% By Carath\'{e}odory's extension theorem, .
% \end{proof}
\begin{lemma}
We have that $\mathfrak{M}^* \vDash \varphi$ iff $\varphi \in \Gamma^*$.
\end{lemma}
\begin{proof}
The axiom $\mathsf{EqReplace}$ guarantees that our definition \eqref{eqn:defmu} in terms of a single representative constant $g_X(i_X) \in \mathcal{C}_X$ is consistent in $\Gamma^*$.
Otherwise, comparing to \citet[Theorem~3]{Perovic}, we have two features to address.
First, $\mathsf{EqDist}$ guarantees the result for formulas that include primitive probabilities $\probsymbol(\delta)$ where literals involving the same variable $V$ occur more than once in $\delta$.

Second is the presence of sums here, for which it suffices to show that $\sanssum_{x_i} \mathsf{t} \succsim \underline{q} \in \Gamma^*$ iff $q \le \semantics{\sanssum_{x_i} \mathsf{t}}^{\mathfrak{M}^*}$, and $\sanssum_{x_i} \mathsf{t} \precsim \underline{q} \in \Gamma^*$ iff $q \ge \semantics{\sanssum_{x_i} \mathsf{t}}^{\mathfrak{M}^*}$, %and $\sanssum_{x_i} \mathsf{t} \precsim \underline{q} \in \Gamma^*$ iff $\semantics{\sanssum_{x_i} \mathsf{t}}^{\mathfrak{M}^*} \le q$
for any term $\mathsf{t}$ and $q \in \mathbb{Q}$.
Suppose $\sanssum_{x_i} \mathsf{t} \succsim \underline{q} \in \Gamma^*$.
If $q > \semantics{\sanssum_{x_i} \mathsf{t}}^{\mathfrak{M}^*}$ then there is some rational $\varepsilon > 0$ such that
\begin{align*}
\semantics{ \mathsf{t}}^{\mathfrak{M}^*}_{\iota[X, i \mapsto 1]} + \dots + \semantics{ \mathsf{t}}^{\mathfrak{M}^*}_{\iota[X, i \mapsto n]} \le q - \varepsilon
\end{align*}
for any $n \in \Val(X)$.
By consistency with $\mathsf{SumUpper}$, this means $\underline{q- \varepsilon} \succsim \sanssum_{x_i} \mathsf{t} \in \Gamma^*$, so $\underline{q- \varepsilon} \succsim \underline{q} \in \Gamma^*$, a contradiction.
Conversely suppose $q \le \semantics{\sanssum_{x_i} \mathsf{t}}^{\mathfrak{M}^*}$. If $\sanssum_{x_i} \mathsf{t} \prec \underline{q} \in \Gamma^*$ then by consistency with $\mathsf{Conv}$ there is some $\varepsilon > 0$ such that $\sanssum_{x_i} \mathsf{t} \prec \underline{q-\varepsilon} \in \Gamma^*$ and by $\mathsf{SumLower}$ we have that $\semantics{ \mathsf{t}}^{\mathfrak{M}^*}_{\iota[X, i \mapsto 1]} + \dots + \semantics{ \mathsf{t}}^{\mathfrak{M}^*}_{\iota[X, i \mapsto n]} < q -\varepsilon$ for any $n \in \Val(X)$. This implies $\semantics{\sanssum_{x_i} \mathsf{t}}^{\mathfrak{M}^*} \le q-\varepsilon$, a contradiction.

Now suppose $\sanssum_{x_i} \mathsf{t} \precsim \underline{q} \in \Gamma^*$. If $q < \semantics{\sanssum_{x_i} \mathsf{t}}^{\mathfrak{M}^*}$ then there is some $n \in \Val(X)$ such that $q < \semantics{ \mathsf{t}}^{\mathfrak{M}^*}_{\iota[X, i \mapsto 1]} + \dots + \semantics{ \mathsf{t}}^{\mathfrak{M}^*}_{\iota[X, i \mapsto n]}$, and by $\mathsf{SumLower}$, we have $\sanssum_{x_i} \mathsf{t} \succ \underline{q} \in \Gamma^*$, a contradiction.
Conversely if $q \ge \semantics{\sanssum_{x_i} \mathsf{t}}^{\mathfrak{M}^*}$, then $q \ge \semantics{ \mathsf{t}}^{\mathfrak{M}^*}_{\iota[X, i \mapsto 1]} + \dots + \semantics{ \mathsf{t}}^{\mathfrak{M}^*}_{\iota[X, i \mapsto n]}$ for any $n \in \Val(X)$, so that $\underline{q} \succsim \sanssum_{x_i} \mathsf{t} \in \Gamma^*$ by $\mathsf{SumUpper}$.
\end{proof}
This completes the proof of Theorem~\ref{thm:axcomplete}, showing $\mathfrak{M}^*$ satisfies $\Gamma^* \supset \Gamma$. Note that $\mathsf{Cond}$ has been used implicitly in working with plain probabilities throughout rather than conditional probabilities.
\end{proof}

% \subsection{Axiomatizability of $do$}

\section{Conclusion}\label{section: conclusion}
Although the present contribution has settled some of the most basic questions about probability logics with summation, many  open questions remain. As a first example, although we have been able to show strong completeness for infinitary proof systems, questions of weak completeness may also be of interest.

In that direction, consider a language $\mathcal{L}_{\star}$, which is just like $\mathcal{L}_{\textnormal{causal}}$, except that we allow no constants at all, and also disallow equalities between free variables. In other words, this language has only (free and bound) variables occurring in atomic expressions $V =v$; it is of interest since Ex.~\ref{example:front} and other uses of $do$-calculus fall within this natural fragment. We noted in Theorem \ref{theorem: not RE} that this language (or any other language sufficient for encoding conditional independence statements with unbounded variables ranges) will be undecidable. Meanwhile, the proof of Theorem \ref{RE-Pi2} can be readily adapted to show the set of satisfiable $\mathcal{L}_{\star}$-expressions is  recursively enumerable. It follows that validity cannot also be recursively enumerable, which in turn implies the following:

%which is the language considered in this paper containing no constants. This is arguably a natural logical environment for work on do-calculus \citep{Pearl1995,Pearl2009}. We close with a limitative result that follows from our Theorems \ref{theorem: not RE} and \ref{RE-Pi2}.

%For instance, weak axiomatability.
%Also, $\mathcal{L}^i_{do}$, i.e. our languages with \emph{only} free variables and no constants.
% \tb{Maybe better to move this to the conclusion and discuss it in the context of open questions?}
\begin{corollary}
There is no weakly complete, recursive axiomatization of $\mathcal{L}_{\star}$ over models with unbounded finite variable ranges. 
\end{corollary}
\noindent Whether any other settings we have investigated here can be given (finitary) weak axiomatizations we leave as a worthy open question. Another natural question about axiomatization concerns the full causal language. As we have seen, the  principles that encode reasoning about pure probability are already rather complex. Causal reasoning---especially about recursive models---introduces another source of compactness failure. So strong axiomatization would again demand infinitary schemes. We leave this work also for a future occasion. 

Several natural open questions also remain concerning complexity. For example, while we know that the satisfiability problem for the general, ``open universe'' setting is undecidable (Theorem \ref{theorem: not RE}), we do not know how undecidable it is, e.g., when variable ranges can be infinite. A reasonable conjecture is that it is complete for $\Sigma^0_1$, but this remains to be shown. 

Finally, moving beyond the systems we have investigated here, natural extensions suggest themselves for investigation. The continuous setting, with integration operators replacing summation operators, would be of particular interest. Languages for integration have been explored by model theorists (see, e.g., \citealt{Bagheri}), though axiomatic questions appear not to have received much attention. Marginalization (both through summation and through integration) is also closely related to issues of \emph{abstraction} of models, a topic that has been of particular interest in the recent causality literature (see \citealt{Bongers,Geiger2024} for connections between marginalization and causal abstraction). A causal-probability logical system capable of reasoning explicitly about abstraction would also be worth investigating.
%\begin{proof}
%The existence of such an axiomatization would imply that validity is r.e.; since satisfiability is also r.e. in the positive class, this would imply that validity is decidable, contradicting Theorem~\ref{theorem: not RE}.

%To see satisfiability over positive models is r.e., the same procedure can be used as in the proof of Theorem~\ref{RE-Pi2}, so long as we initially append to $\varphi$ the statement that $\mathbb{P}(\delta) > 0$, for all $\delta$ appearing in the condition of a conditional probability in $\varphi$. (Where $\sum_{v_i} \mathbb{P}(\delta|\delta^\prime)$ appears in $\varphi$, we append $\mathbb{P}(\delta^\prime[c / v_i]) >0$ to $\varphi$ for each $c$ denoting a distinct element of $\text{Val}(V)$.) % Thus this setting is also r.e. with respect to satisfiability.
%\end{proof}

\newpage
\bibliographystyle{apalike}
\bibliography{refs}{}

\newcommand{\SortNoop}[1]{}
\begin{thebibliography}{}

\bibitem[Abadi and Halpern, 1994]{abadi1994decidability}
Abadi, M. and Halpern, J.~Y. (1994).
\newblock Decidability and expressiveness for first-order logics of
  probability.
\newblock {\em Information and computation}, 112(1):1--36.

\bibitem[Angrist et~al., 1996]{angrist1996identification}
Angrist, J.~D., Imbens, G.~W., and Rubin, D.~B. (1996).
\newblock Identification of causal effects using instrumental variables.
\newblock {\em Journal of the American Statistical Association},
  91(434):444--455.

\bibitem[Bagheri and Pourmahdian, 2009]{Bagheri}
Bagheri, S.-M. and Pourmahdian, M. (2009).
\newblock The logic of integration.
\newblock {\em Archive for Mathematical Logic}, 48:465–492.

\bibitem[Bareinboim et~al., 2022]{bareinboim2022pearl}
Bareinboim, E., Correa, J., Ibeling, D., and Icard, T. (2022).
\newblock On {P}earl's hierarchy and the foundations of causal inference.
\newblock In Geffner, H., Dechter, R., and Halpern, J.~Y., editors, {\em
  Probabilistic and Causal Inference: The Works of Judea Pearl}, pages
  509--556. ACM Books.

\bibitem[Bläser et~al., 2024]{blaeser2024existential}
Bläser, M., Dörfler, J., Liśkiewicz, M., and van~der Zander, B. (2024).
\newblock The existential theory of the reals with summation operators.
\newblock In {\em Manuscript}.

\bibitem[Bongers et~al., 2021]{Bongers}
Bongers, S., Forr\'{e}, P., Peters, J., and Mooij, J.~M. (2021).
\newblock Foundations of structural causal models with cycles and latent
  variables.
\newblock {\em The Annals of Statistics}, 49(5):2885--2915.

\bibitem[D{\"o}rfler et~al., 2024]{dorfler2024probabilistic}
D{\"o}rfler, J., van~der Zander, B., Bl{\"a}ser, M., and Liskiewicz, M. (2024).
\newblock Probabilistic and causal satisfiability: the impact of
  marginalization.
\newblock {\em arXiv preprint arXiv:2405.07373}.

\bibitem[Erickson et~al., 2022]{erickson2022smoothing}
Erickson, J., Van Der~Hoog, I., and Miltzow, T. (2022).
\newblock Smoothing the gap between np and er.
\newblock {\em SIAM Journal on Computing}, (0):FOCS20--102.

\bibitem[Fagin et~al., 1990]{fagin1990logic}
Fagin, R., Halpern, J.~Y., and Megiddo, N. (1990).
\newblock A logic for reasoning about probabilities.
\newblock {\em Information and Computation}, 87(1-2):78--128.

\bibitem[Geiger et~al., 2024]{Geiger2024}
Geiger, A., Ibeling, D., Zur, A., Chaudhary, M., Chauhan, S., Huang, J., Arora,
  A., Wu, Z., D'Oosterlink, K., Goodman, N.~D., Potts, C., and Icard, T.
  (2024).
\newblock Causal abstraction: A theoretical foundation for faithful
  interpretability.
\newblock Manuscript, Pr(AI)$^2$r Group and Stanford University.

\bibitem[Ibeling and Icard, 2020]{ibelingicard2020}
Ibeling, D. and Icard, T. (2020).
\newblock Probabilistic reasoning across the causal hierarchy.
\newblock In {\em Proceedings of AAAI}.

\bibitem[Ibeling and Icard, 2021]{ibeling2021topological}
Ibeling, D. and Icard, T. (2021).
\newblock A topological perspective on causal inference.
\newblock {\em Advances in Neural Information Processing Systems},
  34:5608--5619.

\bibitem[Ibeling and Icard, 2023]{ibeling2023}
Ibeling, D. and Icard, T. (2023).
\newblock Comparing causal frameworks: Potential outcomes, structural models,
  graphs, and abstractions.
\newblock In {\em Advances in Neural Information Processing Systems 36 (NeurIPS
  2023)}, pages 1--12.

\bibitem[Ibeling et~al., 2023]{IBELING2023103339}
Ibeling, D., Icard, T., Mierzewski, K., and Mossé, M. (2023).
\newblock Probing the quantitative–qualitative divide in probabilistic
  reasoning.
\newblock {\em Annals of Pure and Applied Logic}, page 103339.

\bibitem[Li, 2023]{li2023undecidability}
Li, C.~T. (2023).
\newblock Undecidability of network coding, conditional information
  inequalities, and conditional independence implication.
\newblock {\em IEEE Transactions on Information Theory}, 69(6).

\bibitem[Moss{\'e} et~al., 2024]{mosse2022causal}
Moss{\'e}, M., Ibeling, D., and Icard, T. (2024).
\newblock Is causal reasoning harder than probabilistic reasoning?
\newblock {\em The Review of Symbolic Logic}, 17(1):106--131.

\bibitem[Pearl, 1995]{Pearl1995}
Pearl, J. (1995).
\newblock Causal diagrams for empirical research.
\newblock {\em Biometrika}, 82(4):669--710.

\bibitem[Pearl, 2009]{Pearl2009}
Pearl, J. (2009).
\newblock {\em Causality}.
\newblock Cambridge University Press.

\bibitem[Perovi\'{c} et~al., 2008]{Perovic}
Perovi\'{c}, A., Ognjanovi\'{c}, Z., Ra\v{s}kovi\'{c}, M., and Markovi\'{c}, Z.
  (2008).
\newblock A probabilistic logic with polynomial weight formulas.
\newblock In {\em International Symposium on Foundations of Information and
  Knowledge Systems}, pages 239--252.

\bibitem[Schaefer and {\v{S}}tefankovi{\v{c}}, 2023]{schaefer2023beyond}
Schaefer, M. and {\v{S}}tefankovi{\v{c}}, D. (2023).
\newblock Beyond the existential theory of the reals.
\newblock {\em Theory of Computing Systems}, pages 1--32.

\bibitem[Shpitser and Pearl, 2006]{shpitser2012identification}
Shpitser, I. and Pearl, J. (2006).
\newblock Identification of conditional interventional distributions.
\newblock In {\em Proceedings of the Twenty-Second Conference on Uncertainty in
  Artificial Intelligence}, UAI'06, page 437–444, Arlington, Virginia, USA.
  AUAI Press.

\bibitem[Shpitser and Pearl, 2008]{Shpitser}
Shpitser, I. and Pearl, J. (2008).
\newblock Complete identification methods for the causal hierarchy.
\newblock {\em Journal of Machine Learning Research}, 9:1941--1979.

\bibitem[Spirtes et~al., 2000]{spirtes2000causation}
Spirtes, P., Glymour, C.~N., and Scheines, R. (2000).
\newblock {\em Causation, Prediction, and Search}.
\newblock The MIT Press.

\bibitem[Suppes and Zanotti, 1981]{suppes:zan81}
Suppes, P. and Zanotti, M. (1981).
\newblock {When are probabilistic explanations possible?}
\newblock {\em Synthese}, 48:191--199.

\bibitem[van~der Zander et~al., 2023]{van2023hardness}
van~der Zander, B., Bläser, M., and Liśkiewicz, M. (2023).
\newblock The hardness of reasoning about probabilities and causality.
\newblock In Elkind, E., editor, {\em Proceedings of the Thirty-Second
  International Joint Conference on Artificial Intelligence, {IJCAI-23}}, pages
  5730--5738. International Joint Conferences on Artificial Intelligence
  Organization.
\newblock Main Track.

\end{thebibliography}

\end{document}